\documentclass[11pt,  reqno]{amsart}

\usepackage{amsmath,amssymb,amscd,amsthm,amsxtra, esint}
\usepackage[all]{xy}

\allowdisplaybreaks[2]

\newtheorem{theorem}{Theorem} [section]

\newtheorem{lemma}[theorem]{Lemma}
\newtheorem{proposition}[theorem]{Proposition}
\newtheorem{remark}[theorem]{Remark}

\newtheorem{definition}[theorem]{Definition}
\newtheorem{corollary}[theorem]{Corollary}

\DeclareMathOperator*{\supp}{supp}

\newcommand{\noi}{\noindent}
\newcommand{\Z}{\mathbb{Z}}
\newcommand{\R}{\mathbb{R}}

\newcommand{\T}{\mathbb{T}}

\newcommand{\N}{\mathcal{N}}
\newcommand{\RR}{\mathcal{R}}

\newcommand{\TT}{\mathcal{T}}

\newcommand{\Nf}{\mathfrak{N}}
\newcommand{\Rf}{\mathfrak{R}}

\newcommand{\eps}{\varepsilon}

\newcommand{\G}{\Gamma}

\newcommand{\ft}{\widehat}
\newcommand{\wt}{\widetilde}
\newcommand{\cj}{\overline}
\newcommand{\dx}{\partial_x}
\newcommand{\dt}{\partial_t}

\newcommand{\jb}[1]
{\langle #1 \rangle}

\numberwithin{equation}{section}
\numberwithin{theorem}{section}

\begin{document}



\title
[Poincar\'e-Dulac normal form for cubic NLS]
{Poincar\'e-Dulac normal form reduction for unconditional  well-posedness  of  
the periodic cubic NLS}

\author{Zihua Guo, Soonsik Kwon, and Tadahiro Oh}

\address{
Zihua Guo\\
School of Mathematics\\
Institute for Advanced Study\\
Einstein Drive\\ Princeton\\ NJ 08540\\ USA
} 

\email{guo@math.ias.edu}
%

\address{
Soonsik Kwon\\
Department of Mathematical Sciences\\
Korea Advanced Institute of Science and Technology\\
335 Gwahangno \\ 
Yuseong-gu, Daejeon 305-701, Republic of Korea}

\email{soonsikk@kaist.edu}

\address{
Tadahiro Oh\\
Department of Mathematics\\
Princeton University\\
Fine Hall\\
Washington Rd.\\
Princeton, NJ 08544-1000, USA}

\email{hirooh@math.princeton.edu}

\thanks{Z.G. is supported by the National Science Foundation under agreement No. DMS-0635607 and The S. S. Chern Fund. Any opinions, findings and conclusions or recommendations expressed in this material are those of the authors and do not necessarily reflect the views of the National Science Foundation or The S. S. Chern Fund.}

\thanks{S.K. is supported in part by NRF grant 2010-0024017.}



\begin{abstract}

We implement an infinite iteration scheme of Poincar\'e-Dulac normal form reductions
to establish an energy estimate on 
the one-dimensional cubic
nonlinear Schr\"odinger equation (NLS) in $C_tL^2(\T)$,
without using any auxiliary function space.
This allows us to construct weak solutions of NLS in $C_tL^2(\T)$ with initial data in $L^2(\T)$
as limits of classical solutions.
As a consequence of our construction, we  
also prove unconditional well-posedness of NLS in $H^s(\T)$
for $s \geq \frac{1}{6}$.

\end{abstract}

\maketitle

\tableofcontents

\section{Introduction}
\label{SEC:1}

\subsection{Nonlinear Schr\"odinger equation}

We consider the Cauchy problem for the cubic nonlinear Schr\"odinger equation
(NLS) on the one-dimensional torus $\T = \R/\Z$:
\begin{equation}
	\label{NLS1} 
	\begin{cases}
		i u_t - u_{xx} = \pm u |u|^2  \\
		u|_{t= 0} = u_0, 
	\end{cases}~(x, t)  \in \T \times \R,
\end{equation}

\noi
where $u$ is a complex-valued function.
In this paper, we study unconditional well-posedness of \eqref{NLS1}. 
One-dimensional cubic NLS is known to be completely integrable.
However, our argument does not use such integrability structure of \eqref{NLS1}
in an explicit manner.
Moreover, due to the local-in-time nature of our argument, 
it does not matter whether the equation is defocusing (with $-$ sign)
or focusing (with $+$ sign.)
Hence, we simply assume that it is defocusing in the following.

In \cite{BO1}, Bourgain introduced a new weighted space time Sobolev space $X^{s, b}$,
whose norm is given by 
\begin{equation} \label{XSB}
\|u\|_{X^{s, b}(\T\times \R)} 
= \| \jb{\dx}^s \jb{i \dt - \dx^2}^b  (u) \|_{L^2 (\T\times \R)},
\end{equation}

\noi
where $\jb{\,\cdot\,} = 1 +|\cdot|$.
After establishing the periodic $L^4$-Strichartz estimate
\begin{equation} \label{L4}
\|u\|_{L^4_{x, t}} \lesssim \|u\|_{X^{0,\frac{3}{8}}}, 
\end{equation}  
Bourgain proved that \eqref{NLS1} is locally well-posed in $L^2(\T)$.
Thanks to the $L^2$-conservation law, 
this immediately implied global well-posedness of \eqref{NLS1} in $L^2(\T)$.
This well-posedness result is known to be sharp
in view of the ill-posedness results
of \eqref{NLS1} in $H^s(\T)$, $s<0$, by
Burq-G\'erard-Tzvetkov \cite{BGT},
Christ-Colliander-Tao \cite{CCT1, CCT2},
and Molinet \cite{MOLI}.
Recently, there have been several studies 
on constructing solutions of (the renormalized version of) \eqref{NLS1} in larger spaces than $L^2(\T)$.
See Christ \cite{CH1} and Colliander-Oh \cite{CO}.

\medskip

Our main goal in this paper is 
twofold.
\begin{itemize}
\item[(a)]
We construct weak solutions of \eqref{NLS1} with $u_0 \in L^2(\T)$,
by directly establishing an energy estimate in $C([0, T]; L^2)$
{\it without} using any auxiliary function space.

\item[(b)]
We establish a uniqueness statement of  solutions of \eqref{NLS1}.
For this part, we assume sufficient regularity on solutions.
In particular, we assume that a solution $u$ is in $C([0, T]; H^s)$ for some $s\geq \frac{1}{6}$.
\end{itemize}

\noi
First, let us discuss what we mean by solutions in $C([0, T]; L^2)$.
For this purpose, we use the following notions from Christ \cite{CH1, CH2}.

\begin{definition} \label{DEF:1} \rm
A sequence of Fourier cutoff operators is a sequence of Fourier multiplier operators $\{T_N\}_{N\in \mathbb{N}}$
on $\mathcal{D'}(\T)$ with multipliers $m_N:\mathbb{Z} \to \mathbb{C}$ such that (i) $m_N$ has a compact support on $\mathbb{Z}$
for each $N \in \mathbb{N}$, (ii) $m_N$ is uniformly bounded, and 
(iii) $m_N$ converges pointwise to $1$, i.e. $\lim_{N\to \infty} m_N(n) = 1$ for any $n \in \mathbb{Z}$.
\end{definition}

\noi
The following definition is in particular important in making sense of the nonlinearity
$\mathcal{N}(u) := u |u|^2$, 
when a function $u$ is merely in $C([0, T]: L^2(\T))$. 

\begin{definition} \label{DEF:2} \rm
Let $u \in C([0, T]; L^2(\T))$. 
We say that $\mathcal{N}(u)$ exists and is equal to a distribution $w \in \mathcal{D}'(\T\times (0, T))$
if for every sequence $\{T_N\}_{N\in \mathbb{N}}$ of (spatial) Fourier cutoff operators, we have
\begin{equation} \label{NON1}
\lim_{N\to \infty} \mathcal{N}(T_N u) = w
\end{equation}

\noi
in the sense of distributions on $\T\times (0, T)$.
\end{definition}

\begin{definition} \label{DEF:3} \rm
We say that $u  \in C([0, T]; H^s(\T))$  is a weak solution of NLS \eqref{NLS1}
in the extended sense
if (i) $u\big|_{t=0} = u_0$, (ii) the nonlinearity $\mathcal{N}(u)$ exists in the sense of Definition \ref{DEF:2},
and (iii) $u$ satisfies \eqref{NLS1}
in the sense of distributions on $\T\times (0, T)$,
where the nonlinearity $\mathcal{N}(u) = u |u|^2$ is interpreted as above. 
\end{definition}

In the following, we construct weak solutions of \eqref{NLS1} with $u_0 \in L^2(\T)$,
by directly establishing an energy estimate in $C([0, T]; L^2)$
without using any auxiliary function space.
Our first result is the following.

\begin{theorem}[Existence] \label{thm0}
Let $s \geq 0$.
Then, for  $u_0 \in H^s(\T)$, there exists a weak solution $u \in C([0, T]; H^s(\T))$ of NLS \eqref{NLS1}
with initial condition $u_0$
in the sense of Definition \ref{DEF:3},
where the time $T$ of existence depends only on $\|u_0\|_{H^s}$.
Moreover, the solution map is  Lipschitz continuous.
\end{theorem}

\begin{remark} \label{REM:1} \rm
In view of the embedding $H^s(\T) \subset L^3(\T)$ for $s\geq \frac{1}{6}$, 
it follows that when $s\geq \frac{1}{6}$, the solution $u$ in Theorem \ref{thm0}
indeed satisfies NLS \eqref{NLS1} in the usual distributional sense 
as a space-time distribution. Moreover, for each fixed $t \in (0, T)$,  
it satisfies the equation \eqref{NLS1} as a spatial distribution on $\T$.
\end{remark}

\noi
In Theorem \ref{thm0}, the uniqueness holds only as a limit of classical solutions.
This in particular implies that the solutions in Theorem \ref{thm0}
coincide with the solutions constructed in Bourgain's $L^2$ well-posedness result 
\cite{BO1}. 
Hence, they lie in the class 
\begin{equation} \label{Bourgain}
C([0,T];L^2(\T))\cap X_T^{0,\frac{3}{8}}
\subset C([0,T];L^2(\T))\cap L^4_{x, T},
\end{equation}

\noi  
where $X_T^{0,\frac{3}{8}}$ denotes the local-in-time version of $X^{0,\frac{3}{8}}$
and $L^4_{x, T} = L^4 (\T\times [0, T])$.

\medskip
Now, let us turn to the uniqueness statement of  solutions of \eqref{NLS1}.
Recall the following definition from Kato \cite{KATO}.
We say that a Cauchy problem is {\it unconditionally well-posed}
in $H^s$
if for every initial condition $u_0 \in H^s$,
there exist $T>0$ and a {\it unique} solution $u \in C([0, T];H^s)$
such that $u(0) = u_0$.
Also, see \cite{FPT}.
We refer to such uniqueness in $C([0, T];H^s)$
without intersecting with any auxiliary function space as {\it unconditional uniqueness}.
Unconditional uniqueness is a concept of uniqueness which does not depend
on how solutions are constructed.

As mentioned above, the $L^2$ well-posedness result 
in \cite{BO1} assumes that solutions are a priori
in $X^{0,\frac{3}{8}}$ (locally in time)
due to the use of the periodic $L^4$ Strichartz estimate \eqref{L4}.
Thus, the uniqueness in \cite{BO1} holds only  in the class \eqref{Bourgain}.
Namely,  the uniqueness of solutions in \cite{BO1} holds {\it conditionally},
since uniqueness may not hold without the restriction 
of the auxiliary function space $X_T^{0,\frac{3}{8}}$.

The proof of  Theorem \ref{thm0} only uses a direct energy estimate in $C([0, T]; L^2)$.
For a general solution $u \in C([0, T]; L^2)$, we need to perform the argument through
an approximating smooth solution $u_n$.
However, if $u \in C([0, T]; H^s)$ for some $s \geq \frac{1}{6}$, we do not need such an approximating sequence
of smooth solutions and directly work on $u$.
This yields the following uniqueness of  solutions to \eqref{NLS1} in $H^s(\T)$, $s \geq \frac{1}{6}$. 

\begin{theorem}[Unconditional uniqueness] \label{thm1}
Let $s \geq \frac{1}{6}$.
Then, for  $u_0 \in H^s(\T)$, 
the solution $u $ with initial condition $u_0$ 
constructed in Theorem \ref{thm0}
is unique in $C([0, T]; H^s(\T))$.
\end{theorem}

\noi
Theorem \ref{thm1} with the $L^2$-conservation law for \eqref{NLS1}
yields  the following corollary.

\begin{corollary} \label{cor1}
 Let $s \geq \frac{1}{6}$.  NLS \eqref{NLS1}  is unconditionally globally well-posed
in $H^s(\T)$.
\end{corollary}

Note that Theorem \ref{thm1} is an 
improvement of Bourgain's result \cite{BO1} in the aspect of uniqueness, 
at least for $s \geq \frac{1}{6}$.
We also point out that, for $s < 0$, there is the non-uniqueness result 
in $C_tH^s(\T)$ by Christ \cite{CH1} 
(for solutions in extended sense as in Definition \ref{DEF:3}.) 
See \cite{CH1} for details.

Many of the unconditional uniqueness results
use some auxiliary function spaces (such as Strichartz spaces and $X^{s, b}$ spaces),
which are designed to be large enough to contain $C([0, T];H^s)$
such that desired nonlinear estimates hold.
See, for instance, Zhou \cite{Z}.
However,  we simply use
the $C([0, T];H^s)$-norm in the proof of Theorem \ref{thm1}.

For $s > \frac{1}{2}$, 
an a priori estimate in $C([0, T];H^s)$
easily follows from Sobolev embedding theorem.
The challenge is to go below this regularity.
We achieve this goal by implementing an infinite iteration scheme
for the {\it Poincar\'e-Dulac normal form reductions.}
See Subsection \ref{SUBSEC:1.2} for 
a discussion on the Poincar\'e-Dulac normal form reduction.
This method provides a new method to prove well-posedness of PDEs.

\medskip

In \cite{CH2}, Christ used the power series argument
to construct solutions to the (renormalized) cubic NLS
in $C([0, T]; \mathcal{F}L^{s,p})$, $s\geq 0$, $p \in [1, \infty)$, where 
the Fourier-Lebesgue space $\mathcal{F}L^{s,p}$ is defined by the norm
$\|f\|_{\mathcal{F}L^{s,p}} = \|\jb{n}^s \ft{f}(n)\|_{\ell^p_n}$.
This argument involves a power series expansion of solutions
in terms of initial data and uses only  the $C([0, T]; \mathcal{F}L^{s,p})$-norm.
However, this construction does  not provide uniqueness.\footnote{Gr\"unrock-Herr \cite{GH} 
proved (conditional) local well-posedness 
in $\mathcal{F}L^{s,p}$, $s\geq 0$, $p \in (2, \infty)$
via the fixed point argument.  The uniqueness in \cite{GH} holds only in
$C([0, T]; \mathcal{F}L^{s,p})$ intersected with the $X^{s, b}$ space adapted to 
$\mathcal{F}L^{s,p}$.}
When $p = 2$, a slight modification of his argument
can be applied to the original cubic NLS \eqref{NLS1} 
for construction of solutions in  $C([0, T]; H^s)$, $s\geq0$
(without any auxiliary norms.)
It may be of interest to compare 
and possibly combine
two infinite iteration arguments in  \cite{CH2} and this paper.

We prove Theorems \ref{thm0} and  \ref{thm1}
by establishing {\it a priori} estimates,
where we only use  the $C_tH^s_x$-norm of solutions.
In the next subsection, we briefly describe the idea of
Poincar\'e-Dulac normal form reductions.

Before doing so, let us introduce an equivalent formulation to \eqref{NLS1}.
Let $S(t) = e^{-i t\dx^2}$ denote the semigroup to the linear Schr\"odinger equation:
$i u_t - u_{xx} = 0$.
We apply a change of coordinates: $v(t) = S(-t) u(t) =  e^{it\dx^2} {u}(t)$,
i.e. $\ft{v}(n, t) = e^{-in^2t} \ft{u}(n,t)$. 
For simplicity of notation, we use $v_n = v_n(t)$ to denote
$\ft{v}(n, t)$.
Then, $v$ satisfies the following equation:
\begin{align} \label{NLS3}
\dt v_n (t)& = i \Nf(v, v, v)(n, t)\\
 :& = i 
\sum_{n = n_1 - n_2 + n_3 }
e^{- i \Phi(\bar{n})t } 
v_{n_1} \cj{v}_{n_2}v_{n_3}\notag \\
& = 
i \sum_{\substack{n = n_1 - n_2 + n_3 \\ n_2 \ne n_1, n_3} }
e^{- i \Phi(\bar{n})t } 
v_{n_1} \cj{v}_{n_2}v_{n_3}
+ i 
\sum_{\substack{n = n_1 - n_2 + n_3 \\ n_2 = n_1 \text{ or }n_3} }e^{- i \Phi(\bar{n})t } 
v_{n_1} \cj{v}_{n_2}v_{n_3}\notag \\
 & =: i\,  \Nf_1(v, v, v)(n, t) + i \Rf_1(v, v, v)(n, t).
 \end{align}

\noi
Note that $v(0) = u_0 \in H^s(\T)$.
The phase function $\Phi(\bar{n})$ is defined by 
\begin{align}\label{Phi}
\Phi(\bar{n}):& = \Phi(n, n_1, n_2, n_3) = n^2 - n_1^2 + n_2^2- n_3^2 \notag \\
& = 2(n_2 - n_1) (n_2 - n_3)
= 2(n - n_1) (n - n_3),
\end{align}

\noi
where the last two equalities hold under $n = n_1 - n_2 + n_3$.
From \eqref{Phi}, it follows that 
$\Nf_1$ corresponds to the non-resonant part (i.e. $\Phi(\bar{n})\ne 0$) of the nonlinearity.
Throughout this paper, we introduce several multilinear expressions.\footnote{By a multilinear operator, 
we mean an operator such that it is linear or conjugate linear with respect to each argument.}
We often suppress its dependence on $t$
and its multiple arguments $v$.
For example, we write 
 $(\Nf_1)_n$  or $\Nf_1(v)_n$ for $\Nf_1(v, v, v)(n, t)$.

As is well known, 
working in terms of $v$ has certain advantages.
In \cite{BO1}, Bourgain made an effective use of this coordinate
(called interaction representation in Quantum Mechanics \cite{G})
by introducing the $X^{s, b}$ spaces.
From the definition \eqref{XSB}, we have
$\|u\|_{X^{s, b}} = \|v\|_{H^b_tH^s_x}$,
i.e. a function $u$ is in $X^{s, b}$
if and only if its interaction representation $v(t) = S(-t) u(t)$
is in the classical Sobolev space  $H^b_t H^s_x$.

\subsection{Poincar\'e-Dulac normal form reduction: formal argument}

\label{SUBSEC:1.2}

First, let us describe the classical {\it Poincar\'e-Dulac Theorem}.
Consider a formal vector-valued power series $Ax + F(x) := Ax + \sum_{j = a}^\infty f_j(x)$,
with some $a \geq 2$, 
in $n$ variables $x = (x_1, x_2, \cdots, x_n)$, 
where $f_j(x)$ denotes nonlinear terms of degree $j$ in $x$.
Assume that the eigenvalues of $A$ are distinct.
Then, 
{\it Poincar\'e-Dulac Theorem} \cite{A} states the following.
Given 
a differential equation 
\begin{equation}\label{PD1}
\dt x = Ax + F(x) = Ax  + \sum_{j = a}^\infty f_j(x),
\end{equation}

\noi
we can introduce a sequence of formal changes of variables 
\begin{align} 
  & z_1  = x + y_1,   \notag \\
  & z_2  = z_1 + y_2 = x+ y_1 + y_2, \notag \\
  & \hphantom{XX} \vdots \notag \\
  & z  = z_\infty = x + \sum_{j = 1}^\infty y_j, \label{PD5}
\end{align}
  
\noi
to reduce the system
to the canonical form:
\begin{equation}\label{PD2} 
\dt z = A z + G(z)
= Az + \sum_{j = a}^\infty g_j(z),
\end{equation}

\noi
where $g_j(z)$  in the series $G(z) = \sum_{j = a}^\infty g_j(z) $
denotes
 {\it resonant} monomials of degree $j$ in $z$.\footnote{In this formal discussion,
 we intentionally remain vague about the definition of resonant monomials.
 See Arnold \cite{A} for the precise definition.}
Moreover, after the $k$th step, we have
\begin{equation}\label{PD6} 
\dt z_k = A z_k + G_k (z_k),
\end{equation}

\noi
where 
monomials of degree up to $k(a-1)+a-2$  in $G_k (z_k)$ are all resonant.

With $\wt{x}(t) = e^{-tA} x(t)$, 
$\wt{y}_j(t) = e^{-tA} y_j(t)$, and so on, namely ``working in terms of the  interaction representations,'' 
we can rewrite the original equation \eqref{PD1}
as
\begin{equation} \label{PD3}
\dt \wt{x} = e^{-tA} F(e^{tA} \wt{x}),
\end{equation}

\noi
and the resulting canonical equations  \eqref{PD6} and \eqref{PD2} as
\begin{equation} \label{PD4}
\begin{cases}
\dt \wt{z}_k =  e^{-tA} G_k (e^{tA} \wt{z}_k), & \text{after the $k$th step,}\\
\dt \wt{z} = e^{-tA} G(e^{tA} \wt{z}), & k = \infty.
\end{cases}
\end{equation}

\noi
Note that the right hand sides of \eqref{PD4} 
consist of only resonant monomials when $k = \infty$
(and up to degree $k(a-1) + a-2$ after the $k$ th step.)
After integrating  \eqref{PD4} in time, we obtain
\begin{equation} \label{PD9}
\begin{cases}  
\wt{z}_k (t) = \wt{z}_k(0) + \int_0^t  e^{-t'A} G_k (e^{t'A} \wt{z}_k(t')) dt',  
& \text{after the $k$th step,}\\
\wt{z}(t)  = \wt{z}(0) +\int_0^t e^{-t'A} G(e^{t'A} \wt{z}(t'))dt',  
& k = \infty.
\end{cases}
\end{equation}

\noi
With \eqref{PD5},  we formally have
\begin{enumerate}
\item After the $k$th step:
\begin{equation} \label{PD7}
\wt{x}(t) = \wt{x}(0) - \sum_{j = 1 }^k \big[\,  \wt{y}_j(t) 
-\wt{y}_j(0) \big]+ \int_0^t e^{-t'A} G_k (e^{t'A} \wt{z}_k (t'))dt'.
\end{equation}

\noi
Recall that
monomials of degree up to $k(a-1)+a-2$ in $G_k (z_k)$ are all resonant.

\item With $k = \infty$:
\begin{equation} \label{PD8}
\wt{x}(t) = \wt{x}(0) - \sum_{j = 1}^\infty \big[\, \wt{y}_j(t) 
- \wt{y}_j(0) \big]+ \int_0^t e^{-t'A} G (e^{t'A} \wt{z}(t')) dt'.
\end{equation}
\end{enumerate}

\noi
The point of the classical Poincar\'e-Dulac normal form
is to renormalize the flow so that it is expressed in terms of resonant terms as in \eqref{PD6}, \eqref{PD4},
and \eqref{PD9}.
However, for our purpose, the formulations \eqref{PD7} and \eqref{PD8}
turn out to be more useful.

\medskip

In the following, we take the infinite dimensional system \eqref{NLS3},
and formally apply the Poincar\'e-Dulac normal form reductions
to it.
In order to prove Theorems \ref{thm0} and \ref{thm1}, 
we present the revised application of  
the Poincar\'e-Dulac normal form reductions
{\it with estimates} in Sections \ref{SEC:2} and  \ref{SEC:3}.

The term $\Rf_1(v)$ in \eqref{NLS3} consists of only resonant monomials, 
and thus we leave it as it is.
Now, apply {\it differentiation by parts}, i.e. integration by parts 
without an integration symbol - this terminology was introduced in 
Babin-Ilyin-Titi \cite{BIT} -
on the non-resonant part $\Nf_1(v)$:
\begin{align}
(\Nf_1(v))_n & = \dt \bigg[ i \sum_{\substack{n = n_1 - n_2 + n_3 \\ n_2 \ne n_1, n_3} }
\frac{e^{- i \Phi(\bar{n})t } }{ \Phi(\bar{n})}
v_{n_1} \cj{v}_{n_2}v_{n_3}\bigg] \notag \\
&\hphantom{X} - i \sum_{\substack{n = n_1 - n_2 + n_3 \\ n_2 \ne n_1, n_3} }
\frac{e^{- i \Phi(\bar{n})t } }{ \Phi(\bar{n})}
\dt\big( v_{n_1} \cj{v}_{n_2}v_{n_3}\big) \notag\\
& =: \dt(\Nf_{21})_n + (\Nf_{22})_n. \label{fN1}
\end{align}

\noi
For simplicity of presentation, let us drop the complex number $i$
and simply use $1$ for $\pm 1$ and $\pm i$
appearing in the following formal computation.
Moreover, assume that the time derivative falls on the first factor
$v_{n_1}$
of $v_{n_1} \cj{v}_{n_2}v_{n_3}$ in the second term $\Nf_{22}$ in \eqref{fN1},
counting the multiplicity.
Then, from \eqref{NLS3}, we have
\begin{align}
(\Nf_{22})_n 
& = 3 \sum_{\substack{n = n_1 - n_2 + n_3 \\ n_2 \ne n_1, n_3} }
\frac{e^{- i \Phi(\bar{n})t } }{ \Phi(\bar{n})}
\,  (\Nf)_{n_1}
 \cj{v}_{n_2}v_{n_3}\notag\\
& = 
3 \sum_{\substack{n = n_1 - n_2 + n_3 \\ n_2 \ne n_1, n_3} }
\sum_{n_1 = m_1 - m_2 + m_3 }
\frac{e^{- i (\Phi(\bar{n})+ \Phi(\bar{m}))t } }{ \Phi(\bar{n})}
v_{m_1} \cj{v}_{m_2}v_{m_3} \cj{v}_{n_2}v_{n_3} \label{fN22} .
\end{align}

\noi
As before, the phase function $\Phi(\bar{m})$ is defined by 
\begin{align}\label{Phi_m}
\Phi(\bar{m}):& = \Phi(n_1, m_1, m_2, m_3) = n_1^2 - m_1^2 + m_2^2- m_3^2 \notag \\
& = 2(m_2 - m_1) (m_2 - m_3)
= 2(n_1 - m_1) (n_1 - m_3),
\end{align}

\noi
where the last two equalities hold under $n_1 = m_1 - m_2 + m_3$.

In particular, $\Nf_{22}$ consists of quintic monomials.
Then, from \eqref{NLS3},  \eqref{fN1}, and \eqref{fN22}, we have
\begin{align} \label{v1}
v(t) = v(0) + \Nf_{21}(t) - \Nf_{21}(0)
+ \int_0^t \Rf_1(t') + \Nf_{22}(t') dt'.
\end{align}

\noi
This corresponds to \eqref{PD7} with $k = 1$ (and $ a = 3$.)
Indeed,  all of the cubic monomials in the integrand of \eqref{v1}
are in $\Rf_1$, and they are all resonant, verifying the condition $3 = k + a - 1$
with $k = 1$ and $a = 3$.
Also, $\Nf_{21}$ corresponds to the first correction term $\wt{y_1}$
and its degree is 3.

In the second step, we can divide $\Nf_{22}$ into 
its resonant part $\Rf_1$ and non-resonant part $\Nf_2$,
according to $\Phi(\bar{n})+ \Phi(\bar{m}) = 0$ or $\ne 0$.
Then, we apply differentiation by parts on 
the non-resonant part $\Nf_2$.
This yields
\begin{equation} \label{fN2}
\Nf_2 = \dt \Nf_{31} + \Nf_{32}
\end{equation}

\noi
where $\Nf_{31}$ consists of quintic monomials
and $\Nf_{32}$ consists of septic monomials.
Moreover, the constant appearing in front of the summation is $3\cdot 5$.
(We assume that the time derivative
falls on the first of the five factors,
and thus we need to count the multiplicity.)  See \eqref{fN22}. 
Then, from \eqref{NLS3},  \eqref{fN1}, and \eqref{fN2}, we have
\begin{align} \label{v2}
v(t) = v(0) + \sum_{j = 1}^2 \big[\, \Nf_{(j+1) 1}(t) - \Nf_{(j+1)1}(0)\big]
+ \int_0^t \Rf_1(t') + \Rf_2(t') + \Nf_{32}(t') dt',
\end{align}

\noi
corresponding to \eqref{PD7} with $k = 2$ (and $ a = 3$.)
Since $\Nf_{32}$ consists of septic terms, 
all the terms up to degree 5 in the integrand in \eqref{v2}
are in $\Rf_1$ or $\Rf_2$, and hence they are resonant.

\medskip

In this way, we can repeat this formal procedure indefinitely. 
However, for our purpose, we need to estimate each term in $C_t H^s$,
and there are three potential difficulties.
\begin{enumerate}
\item  We need to estimate higher and higher order monomials.
This corresponds to establishing multilinear estimates
with higher and higher degrees of nonlinearities.

\item At the $k$th step, 
the number of factors on which the time derivative falls
is $2k+1$.
Thus, the constants grow like $3\cdot5\cdot7\cdot \cdots \cdot (2k+1)$.

\item Our multilinear estimates need to provide small constants
on the terms {\it without} time integration,
i.e. on the boundary terms, such as 
$ \Nf_{21}(t) - \Nf_{21}(0)$ in \eqref{v1}
and $\sum_{j = 1}^2 \big[\, \Nf_{(j+1)1}(t) - \Nf_{(j+1)1}(0)\big]$
in \eqref{v2}.
(We can introduce small constants for the terms
inside time integration by making the time interval of integration 
sufficiently small, depending on $\|u_0\|_{L^2}$.)

\end{enumerate}

\noi
In the following two sections, 
we revise this formal iteration of Poincar\'e-Dulac normal form reductions
to treat these three issues.
In particular, 
when we apply differentiation by parts on the non-resonant part $\Nf_k$
consisting of monomials of degree $2k+1$, 
we first divide it into two parts:
a part on which we can directly establish $(2k+1)$-linear estimate
(without differentiation by parts)
and a part on which we can not establish any $(2k+1)$-linear estimate.
Then, we apply differentiation by parts on the second part.
The issues (2) and (3) can be treated by 
introducing different levels of thresholds
for separating resonant and non-resonant parts
at each iteration step.
See \eqref{A_N}, \eqref{C1}, \eqref{C2}, and \eqref{CJ}.
Lastly, we point out that the crucial tool
for establishing multilinear estimates
is the divisor counting argument
(as in the proof of the periodic $L^4$- and $L^6$-Strichartz estimates
by Bourgain \cite{BO1}.) See \eqref{divisor}.

A precursor to this argument appears in the work by
Babin-Ilyin-Titi \cite{BIT} for KdV on $\T$,
followed by the authors \cite{KO} for mKdV on $\T$.\footnote{This kind of integration by parts
was previously used in Takaoka-Tsutsumi \cite{TT}.}
Note that two iterations were sufficient in \cite{BIT, KO}
(in \cite{KO}, the second differentiation by parts
is performed  in a slightly different manner
in the endpoint case)
whereas, for cubic NLS, we need to iterate the argument infinitely many times.
This is perhaps due to 
weaker dispersion of the Schr\"odinger equation
as compared to that of 
the Airy equation (= linear part of KdV and mKdV.)
Also,  Shatah \cite{SHA}
and, more recently, Germain-Masmoudi-Shatah \cite{GMS}
use ideas from Poincar\'e-Dulac normal form reduction\footnote{In \cite{GMS}, they introduced
time resonances, space resonances, and space-time resonances.
Resonances in this paper correspond to their time resonances.}
(to send a quadratic term into a cubic one by one iteration.)
However, their goal is global-in-time behavior of small solutions,
and is different from ours
(local-in-time with large data.)

Note that the Poincar\'e-Dulac normal form can be (formally) applied to 
non-Hamiltonian equations, whereas the Birkhoff normal form 
is for Hamiltonian equations.
See Bourgain \cite{BO2, BO3} and Colliander-Kwon-Oh \cite{CKO}
for inductive argument on the application of the Birkhoff normal form.
We point out that the argument in \cite{BO2, BO3, CKO}
is for large times with finite numbers of iterations,
whereas our argument is local-in-time with an infinite number of iterations.

\medskip

This paper is organized as follows.
In Section \ref{SEC:2}, we present the first step of (a revised version of)
Poincar\'e-Dulac normal form reduction
along with relevant estimates.
In Section \ref{SEC:3},
we introduce some notations and implement an infinite iteration scheme
of (a revised version of) Poincar\'e-Dulac normal form reductions,
establishing estimates on the terms appearing at each step.
In Section \ref{SEC:4},
we first express a smooth solution as
a sum of 
infinite series (see \eqref{41}),
and make sense of such a representation 
by the estimates in Sections \ref{SEC:2} and \ref{SEC:3}.
Then, we construct a weak solution in $C([0, T];L^2)$
with initial condition in $L^2$.
In Section \ref{SEC:5}, we work on $H^s$ for $s\geq \frac{1}{6}$
and justify the formal argument in Sections \ref{SEC:2} and \ref{SEC:3}.
This proves unconditional uniqueness in $C([0, T];H^s)$ for $s\geq\frac{1}{6}$.

\section{Poincar\'e-Dulac normal form reduction,  Part 1: basic setup}
\label{SEC:2}

In this section, we discuss the first step of Poincar\'e-Dulac normal form reduction.
In the following, we take $s = 0$ for simplicity, and 
 estimate each multilinear expression appearing in the discussion
by the $L^2_x$-norm, independent of time.
Namely, we establish direct $C_t L^2_x$ estimates.
Then, we implement an infinite iteration scheme
in the next section.
As in Section \ref{SEC:1}, we often drop the complex number $i$
and simply use $1$ for $\pm 1$ and $\pm i$
in the following.\footnote{When we apply differentiation by parts,
we keep the minus sign on the second term
for emphasis.
For example, see \eqref{N12}.}
Lastly, in Sections \ref{SEC:2} and \ref{SEC:3},
we perform all the formal computations, assuming that $u$
(and hence $v$)
is a smooth solution. 
In Section \ref{SEC:5}, we justify our formal computations
when $ u \in C_t H^s$, $ s\geq \frac{1}{6}$.

First, we write the nonlinearity $u|u|^2$ in \eqref{NLS1}
as
\begin{align*}
u|u|^2 & = \bigg( {u} |{u}|^2 - 2{u} \fint_\T \ |{u}|^2 dx\bigg) + 2{u} \fint_\T \ |{u}|^2 dx\notag \\
& = \sum_{n_2 \ne n_1, n_3} \ft{{u}}(n_1)\cj{{{u}}(n_2)}\ft{{u}}(n_3) 
	e^{i(n_1 - n_2 + n_3)x} - 
	\sum_n \ft{{u}}(n)|\ft{{u}}(n)|^2 e^{inx}\\
& \hphantom{X}	
+ 2\bigg(\fint_\T \ |{u}|^2 dx\bigg) \sum_n \ft{u}(n) e^{inx},
\end{align*}

\noi
where $\fint_\T |u|^2 dx := \frac{1}{2\pi} \int_\T |u|^2 dx$.
Then, \eqref{NLS3} can be written as 
\begin{align} \label{NLS4}
\dt v_n & = i 
\sum_{\substack{n = n_1 - n_2 + n_3\\ n_2\ne n_1, n_3} }
e^{- i \Phi(\bar{n})t } 
v_{n_1} \cj{v}_{n_2}v_{n_3}
- i|v_n|^2 v_n  
+2 i \bigg(\fint_\T \ |v|^2 dx\bigg) v_n \notag \\
& =: i\,  \N_1(v)(n) -i\,  \RR_1(v)(n) + i \RR_2(v)(n),
 \end{align}

\noi
where the phase function $\Phi(\bar{n})$ is as in \eqref{Phi}.
From \eqref{Phi}, it follows that 
$\N_1$ corresponds to the non-resonant part (i.e. $\Phi(\bar{n})\ne 0$) of the nonlinearity
and $\RR_1$ and $\RR_2$ correspond to the resonant part.


\begin{lemma}\label{LEM:R1}
Let $\RR_1$ and $\RR_2$ be as in \eqref{NLS4}.
Then, we have
\begin{align} \label{R1_1}
\| \RR_j(v)\|_{L^2} & \lesssim \|v\|_{L^2}^3,\\
\| \RR_j(v) - \RR_j(w)\|_{L^2} 
& \lesssim 
\big(\|v\|_{L^2}^2 + \|w\|_{L^2}^2\big) \|v - w\|_{L^2} \label{R1_2}
\end{align}

\noi
for $j = 1, 2$.
\end{lemma}

\begin{proof}
For $\RR_1$, this is clear from $\ell^2_n \subset \ell^6_n$.
For $\RR_2$, the result follows from Cauchy-Schwarz inequality, once we note
\begin{align*}
2\bigg(\int_\T &  |v|^2 dx\bigg) v_n
- 2\bigg(\int_\T  |w|^2 dx\bigg) w_n\\
& = 2\bigg(\int_\T v (\cj{v} - \cj{w}) dx + \int_\T (v - w)\cj{w} dx\bigg) v_n
+2 
\bigg(\int_\T \ |v|^2 dx\bigg) (v_n -w_n). \qedhere
\end{align*}

\end{proof}

\medskip

Next, we consider the non-resonant part $\N_1$.
Let $N >0$ be a large parameter.
(As we see later, $N = N(\|u_0\|_{L^2})$.)
First, we write 
\begin{equation} \label{N1}
\N_1 = \N_{11} + \N_{12},
\end{equation}

\noi
where $\N_{11}$ is the restriction of $\N_1$
onto $A_N$, where $A_N = \bigcup_n A_N(n)$ with
\begin{align}\label{A_N}
A_N(n):= \big\{ (n, n_1, n_2, n_3); & \ n = n_1 - n_2 + n_3, \ n_1, n_3 \ne n, \notag\\
& |\Phi(\bar{n})| = |2(n - n_1) (n - n_3)| \leq N \big\}
\end{align}

\noi
and $\N_{12} := \N_1 - \N_{11}$.

Recall the following number theoretic fact \cite{HW}.
Given an integer $m$, let $d(m)$ denote the number of divisors of $m$.
Then, we have
\begin{equation} \label{divisor}
d(m) \lesssim e^{c\frac{\log m}{\log\log m} }
(= o(m^\eps) \text{ for any }\eps>0.)
\end{equation}

\noi
With \eqref{divisor}, we  estimate $\N_{11}$ as follows.
 
\begin{lemma}\label{LEM:N11}
Let $\N_{11}$ be as above.
Then, we have
\begin{align} \label{N11_1}
\| \N_{11}(v)\|_{L^2} & \lesssim N^{\frac{1}{2}+} \|v\|_{L^2}^3,\\
\| \N_{11}(v) - \N_{11}(w)\|_{L^2} 
& \lesssim 
N^{\frac{1}{2}+} \big(\|v\|_{L^2}^2 + \|w\|_{L^2}^2\big) \|v - w\|_{L^2}. \label{N11_2}
\end{align}
\end{lemma}

\begin{proof}
We only prove \eqref{N11_1} since \eqref{N11_2} follows in a similar manner.
Fix $n, \mu \in \Z$ with $|\mu| \leq N$.
Then, from \eqref{divisor}, 
there are at most $o(N^{0+})$ many choices for $n_1$ and $n_3$ (and hence for $n_2$ from $n = n_1 - n_2 + n_3$)
satisfying
\begin{equation}
\label{muu}
\mu = 2(n - n_1) (n - n_3).
\end{equation}

\noi
Then, by Cauchy-Schwarz inequality, we have
\begin{align*}
\|\N_{11}\|_{L^2} 
& = \bigg(\sum_n \Big|
\sum_{|\mu|\leq N } \sum_{\substack{n = n_1 - n_2 + n_3\\ n_2\ne n_1, n_3\\\mu = \Phi(\bar{n})} }
v_{n_1} \cj{v}_{n_2}v_{n_3}
\Big|^2\bigg)^\frac{1}{2} \\
& \leq \bigg\{ \sum_n \bigg(\sum_{|\mu|\leq N} N^{0+} \bigg)
\bigg(\sum_{
\substack{n = n_1 - n_2 + n_3\\n_1, n_3}} |v_{n_1}|^2 
|v_{n_1 + n_3-n}|^2|v_{n_3}|^2\bigg)
\bigg\}^\frac{1}{2}\\
& \lesssim N^{\frac{1}{2}+} \|v\|_{L^2}^3. 
\qedhere
\end{align*}
\end{proof}

\medskip
Now, we apply (the first step of) Poincar\'e-Dulac normal form reduction
to the remaining non-resonant part $\N_{12}$.
Namely, we differentiate $\N_{12}$ by parts 
(i.e. apply the product rule on differentiation in a reversed order)
and write
\begin{align}
 \N_{12}(v)_n  & =  
  \sum_{ A_N(n)^c} 
\dt\bigg( \frac{e^{- i \Phi(\bar{n})t } }{-i\Phi(\bar{n})}\bigg)
v_{n_1} \cj{v}_{n_2}v_{n_3} \notag \\
 & = 
i  \sum_{
 A_N(n)^c} 
 \dt \bigg[
\frac{e^{- i \Phi(\bar{n})t } }{2(n - n_1) (n - n_3)}
v_{n_1} \cj{v}_{n_2}v_{n_3}\bigg]  \notag \\
& \hphantom{XXX} 
-i  \sum_{
A_N(n)^c} 
\frac{e^{- i \Phi(\bar{n})t } }{2(n - n_1) (n - n_3)}
\dt\big( v_{n_1} \cj{v}_{n_2}v_{n_3}\big) \notag \\
 & = 
i \dt \bigg[
 \sum_{
 A_N(n)^c} 
\frac{e^{- i \Phi(\bar{n})t } }{2(n - n_1) (n - n_3)}
v_{n_1} \cj{v}_{n_2}v_{n_3}\bigg]  \notag \\
& \hphantom{XXX} 
-i  \sum_{
A_N(n)^c} 
\frac{e^{- i \Phi(\bar{n})t } }{2(n - n_1) (n - n_3)}
\dt\big( v_{n_1} \cj{v}_{n_2}v_{n_3}\big) \notag \\
& =: \dt (\N_{21})_n + (\N_{22})_n. \label{N12}
\end{align}

\noi
Note that we formally exchanged the order of the  sum and the time differentiation
in the first term
at the third equality.
See Section \ref{SEC:5} for more on this issue.

In the following, we assume that the frequencies $(n, n_1, n_2, n_3)$
are on $A_N^c$ defined in \eqref{A_N}, and we may not state it explicitly.

\begin{lemma}\label{LEM:N21}
Let $\N_{21}$ be as in \eqref{N12}.
Then, we have
\begin{align} \label{N21_1}
\| \N_{21}(v)\|_{L^2} & \lesssim N^{-\frac{1}{2}+} \|v\|_{L^2}^3,\\
\| \N_{21}(v) - \N_{21}(w)\|_{L^2} 
& \lesssim 
N^{-\frac{1}{2}+} \big(\|v\|_{L^2}^2 + \|w\|_{L^2}^2\big) \|v - w\|_{L^2}. \label{N21_2}
\end{align}
\end{lemma}

\begin{proof}
We only prove \eqref{N21_1} since \eqref{N21_2} follows in a similar manner.
On $A_N^c$, we have $|\mu| > N$
where $\mu$ is as in \eqref{muu}.
As before, for fixed $n, \mu \in \mathbb{Z}$, 
there are at most $o(|\mu|^{0+})$ many choices for $n_1$ and $n_3$.
Then, by Cauchy-Schwarz inequality, we have
\begin{align*}
\|\N_{21}\|_{L^2} 
& \lesssim  \bigg\{\sum_n 
 \bigg(\sum_{|\mu| >N}\frac{1}{|\mu|^2} |\mu|^{0+}\bigg)
 \bigg( \sum_{n_1, n_3}
|v_{n_1}|^2 |\cj{v}_{n_2}|^2|v_{n_3}|^2\bigg)
\bigg\}^\frac{1}{2}\\
& \lesssim N^{-\frac{1}{2}+} \|v\|_{L^2}^3. 
\qedhere
\end{align*}
\end{proof}


By symmetry between $n_1$ and $n_3$, we can write the 
remaining term $\N_{22}$ as 
\begin{align}
(\N_{22})_n
& = -  2 i \sum_{\substack{n = n_1 - n_2 + n_3\\ n_2\ne n_1, n_3 }}
\frac{e^{- i \Phi(\bar{n})t } }{2(n - n_1) (n - n_3)}
\, \dt v_{n_1} \cj{v}_{n_2}v_{n_3} \notag \\
& \hphantom{xll} - i   \sum_{\substack{n = n_1 - n_2 + n_3\\ n_2\ne n_1, n_3 }}
\frac{e^{- i \Phi(\bar{n})t } }{2(n - n_1) (n - n_3)}
 v_{n_1} \dt \cj{v}_{n_2}v_{n_3} \notag \\
&  =: (\N_{221})_n + (\N_{222})_n. \label{N22}
\end{align}

\noi
In the following, we only estimate the first term $\N_{221}$
since $\N_{222}$ can be estimated analogously.
From \eqref{NLS4}, $\N_{221}$ can be divided into two terms:
\begin{align} \label{N221}
(\N_{221})_n & = 
  2 \sum_{\substack{n = n_1 - n_2 + n_3\\ n_2\ne n_1, n_3} }
\sum_{\substack{n_1 = m_1 - m_2  +m_3\\m_2 \ne m_1, m_3} }
\frac{e^{- i( \Phi(\bar{n})+ \Phi(\bar{m}) )t } }{2(n - n_1) (n - n_3)}
v_{m_1}\cj{v}_{m_2}v_{m_3}\cj{v}_{n_2}v_{n_3} \notag \\
& \hphantom{X} -   2  \sum_{\substack{n = n_1 - n_2 + n_3\\ n_2\ne n_1, n_3} }
\frac{e^{- i \Phi(\bar{n})t } }{2(n - n_1) (n - n_3)}
(\RR_{1}-\RR_2)_{n_1} \cj{v}_{n_2}v_{n_3}\notag\\
& =: (\N_{3})_n + (\N_{4})_n,
\end{align}

\noi
where the phase function $\Phi(\bar{m})$ is as in \eqref{Phi_m}.
%
%
The second term $\N_4$ can be easily estimated. 

\begin{lemma}\label{LEM:N4}
Let $\N_{4}$ be as in \eqref{N221}.
Then, we have
\begin{align} \label{N4_1}
\| \N_{4}(v)\|_{L^2} & \lesssim N^{-\frac{1}{2}+} \|v\|_{L^2}^5,\\
\| \N_{4}(v) - \N_{4}(w)\|_{L^2} 
& \lesssim 
N^{-\frac{1}{2}+} \big(\|v\|_{L^2}^4 + \|w\|_{L^2}^4\big) \|v - w\|_{L^2}. \label{N4_2}
\end{align}
\end{lemma}

\begin{proof}
This lemma  follows from Lemmata \ref{LEM:N21} and \ref{LEM:R1}.
\end{proof}

Now, it remains to estimate $\N_3$.
As in \eqref{N1}, we separate $\N_3$ into two parts, 
depending on the size of the phase $\Phi(\bar{n})+\Phi(\bar{m})$  (see \eqref{N^2_1} below),
and estimate a part of $\N_3$, 
corresponding to ``small'' phase $\Phi(\bar{n})+\Phi(\bar{m})$,
as in Lemma \ref{LEM:N11}.
See Lemma \ref{LEM:N^2_1}.
Then, we apply (the second step of) Poincar\'e-Dulac normal form reduction
to the remaining (non-resonant) part
with ``large'' phase $\Phi(\bar{n})+\Phi(\bar{m})$.  See \eqref{N^3}.
However, it turns out that in order to prove Theorems \ref{thm0} and  \ref{thm1},
we need to iterate this procedure infinitely many times.
Hence, in the next section, we first set up a necessary machinery 
and perform such an infinite iteration
to estimate $\N_3$.

\section{Poincar\'e-Dulac normal form reduction, Part 2: infinite iteration}
\label{SEC:3}

\subsection{Notations: index by trees}
In this section, we apply Poincar\'e-Dulac normal form reductions
infinitely many times to estimate 
\begin{align} \label{N3}
(\N_3)_n = 2 \sum_{\substack{n = n_1 - n_2 + n_3\\ n_2\ne n_1, n_3} }
\sum_{\substack{n_1 = m_1 - m_2  +m_3\\m_2 \ne m_1, m_3} }
\frac{e^{- i( \Phi(\bar{n})+ \Phi(\bar{m}) )t } }{2(n - n_1) (n - n_3)}
v_{m_1}\cj{v}_{m_2}v_{m_3}\cj{v}_{n_2}v_{n_3} .
\end{align}

\noi
In order to do so, we need to set up some notations.
In the following, the complex conjugate signs on  $v_{n_j}$ do not play any significant role,
and thus we drop the complex conjugate sign.
We also assume that all the Fourier coefficients $v_{n_j}$ are non-negative.

\medskip
When we apply differentiation by parts, 
we obtain terms like $\N_{22}$ in \eqref{N12},
where the time derivative may fall on any of the factors $v_{n_j}$.
In general, the structure of such terms can be very complicated, depending on where the time derivative falls.
In the following, we introduce the notion of trees (in particular, of ordered trees in Definition \ref{DEF:tree3})
for indexing such terms and frequencies arising in the general steps of the Poincar\'e-Dulac normal form reductions.
We point out that some of the definitions are similar, but that some are different 
from those in Christ \cite{CH2}. 

\begin{definition} \label{DEF:tree1} \rm
Given a partially ordered set $\TT$ with partial order $\leq$, 
we say that $b \in \TT$ 
with $b \leq a$ and $b \ne a$
is a child of $a \in \TT$,
if  $b\leq c \leq a$ implies
either $c = a$ or $c = b$.
If the latter condition holds, we also say that $a$ is the parent of $b$.

\end{definition}

\noi 
As in Christ \cite{CH2},
our trees in this paper 
refer to a particular subclass of usual trees with the following properties:
\begin{definition} \label{DEF:tree2} \rm
A tree $\TT$ is a finite partially ordered set satisfying
the following properties.
\begin{enumerate}
\item[(i)] Let $a_1, a_2, a_3, a_4 \in \TT$.
If $a_4 \leq a_2 \leq a_1$ and  
$a_4 \leq a_3 \leq a_1$, then we have $a_2\leq a_3$ or $a_3 \leq a_2$,

\item[(ii)]
A node $a\in \TT$ is called terminal, if it has no child.
A non-terminal node $a\in \TT$ is a node 
with  exactly three children denoted by $a_1, a_2$, and $a_3$,

\item[(iii)] There exists a maximal element $r \in \TT$ (called the root node) such that $a \leq r$ for all $a \in \TT$.
We assume that the root node is non-terminal,

\item[(iv)] $\TT$ consists of the disjoint union of $\TT^0$ and $\TT^\infty$,
where $\TT^0$ and $\TT^\infty$
denote  the collections of non-terminal nodes and terminal nodes, respectively.
\end{enumerate}

\end{definition}

\noi
Note that the number $|\TT|$ of nodes in a tree $\TT$ is $3j+1$ for some $j \in \mathbb{N}$,
where $|\TT^0| = j$ and $|\TT^\infty| = 2j + 1$.
Let us denote  the collection of trees in the $j$th generation (i.e. with $j$ parental nodes) by $T(j)$, i.e.
\begin{equation*}
T(j) := \{ \TT : \TT \text{ is a tree with } |\TT| = 3j+1 \}.
\end{equation*}


\noi
Now, we introduce the  notion of ordered trees.

\begin{definition} \label{DEF:tree3} \rm
We say that a sequence $\{ \TT_j\}_{j = 1}^J$ is a chronicle of $J$ generations, 
if 
\begin{enumerate}
\item[(i)] $\TT_j \in {T}(j)$ for each $j = 1, \dots, J$,
\item[(ii)]  $\TT_{j+1}$ is obtained by changing one of the terminal
nodes in $\TT_j$ into a non-terminal node (with three children), $j = 1, \dots, J - 1$.
\end{enumerate}

\noi
Given a chronicle $\{ \TT_j\}_{j = 1}^J$ of $J$ generations,  
we refer to $\TT_J$ as an {\it ordered tree} of the $J$th generation.
We denote the collection of the ordered trees of the $J$th generation
by $\mathfrak{T}(J)$.
Note that the cardinality of $\mathfrak{T}(J)$ is given by 
\begin{equation} \label{cj1}
 |\mathfrak{T}(J)| = 1\cdot3 \cdot 5 \cdot \cdots \cdot (2J-1) =: c_J.
 \end{equation}
\end{definition}

\begin{remark} \rm
Given two ordered trees $\TT_J$ and $\wt{\TT}_J$
of the $J$th generation, 
it may happen that $\TT_J = \wt{\TT}_J$ as trees (namely as graphs) 
according to Definition \ref{DEF:tree2},
while $\TT_J \ne \wt{\TT}_J$ as ordered trees according to Definition \ref{DEF:tree3}.
Namely, the notion of ordered trees comes with associated chronicles;
it encodes not only the shape of a tree
but also how it ``grew''.
Henceforth, when we refer to an ordered tree $\TT_J$ of the $J$th generation, 
it is understood that there is an underlying chronicle $\{ \TT_j\}_{j = 1}^J$.
\end{remark}

\begin{definition} \label{DEF:tree4} \rm
Given an ordered tree $\TT$ (of the $J$th generation for some $J \in \mathbb{N}$), 
we define an index function ${\bf n}: \TT \to \mathbb{Z}$ such that,
\begin{itemize}
\item[(i)] $n_a = n_{a_1} - n_{a_2} + n_{a_3}$ for $a \in \TT^0$,
where $a_1, a_2$, and $a_3$ denote the children of $a$,
\item[(ii)] $\{n_a, n_{a_2}\} \cap \{n_{a_1}, n_{a_3}\} = \emptyset$ for $a \in \TT^0$,

\item[(iii)] $|\mu_1| := |2(n_r - n_{r_1})(n_r - n_{r_3})| >N$, where $r$ is the root node,  
(recall that we are on $A_N^c$ - see \eqref{A_N}),
\end{itemize}

\noi
where  we identified ${\bf n}: \TT \to \mathbb{Z}$ 
with $\{n_a \}_{a\in \TT} \in \mathbb{Z}^\TT$.

We use 
$\mathfrak{N}(\TT) \subset \mathbb{Z}^\TT$ to denote the collection of such index functions ${\bf n}$.

\end{definition}

\begin{remark} \label{REM:terminal}
\rm Note that ${\bf n} = \{n_a\}_{a\in\TT}$ is completely determined
once we specify the values $n_a$ for $a \in \TT^\infty$.
\end{remark}


\medskip

Given an ordered tree 
$\TT_J$ of the $J$th generation with the chronicle $\{ \TT_j\}_{j = 1}^J$ 
and associated index functions ${\bf n} \in \mathfrak{N}(\TT_J)$,
we would like to keep track of the  ``generations'' of frequencies.
In the following,  we use superscripts to denote such generations of frequencies.

Fix ${\bf n} \in \mathfrak{N}(\TT_J)$.
Consider $\TT_1$ of the first generation.
Its nodes consist of the root node $r$
and its children $r_1, r_2, $ and $r_3$. 
We define the first generation of frequencies by
\[\big(n^{(1)}, n^{(1)}_1, n^{(1)}_2, n^{(1)}_3\big) :=(n_r, n_{r_1}, n_{r_2}, n_{r_3}).\]

\noi
From Definition \ref{DEF:tree4}, we have
\begin{equation*}
 n^{(1)} = n^{(1)}_1 - n^{(1)}_2 + n^{(1)}_3, \quad n^{(1)}_2\ne n^{(1)}_1, n^{(1)}_3.
\end{equation*}

 The ordered tree $\TT_2$ of the second generation is obtained from $\TT_1$ by
changing one of its terminal nodes $a = r_k \in \TT^\infty_1$ for some $k \in \{1, 2, 3\}$
into a non-terminal node.
Then, we define
the second generation of frequencies by
\[\big(n^{(2)}, n^{(2)}_1, n^{(2)}_2, n^{(2)}_3\big) :=(n_a, n_{a_1}, n_{a_2}, n_{a_3}).\]

\noi
Then, we have $n^{(2)} = n_k^{(1)}$ for some $k \in \{1, 2, 3\}$, 
\begin{equation*}
 n^{(2)} = n^{(2)}_1 - n^{(2)}_2 + n^{(2)}_3, \quad n^{(2)}_2\ne n^{(2)}_1, n^{(2)}_3,
\end{equation*}

\noi
where the last identities follow from Definition \ref{DEF:tree4}.

As we see later, this corresponds to introducing a new set of frequencies
after the first differentiation by parts.
For example, in \eqref{N3}, we assumed that the time derivative falls on $v_{n_1^{(1)}}$.
This corresponds to changing the ``first'' child $r_1 \in \TT_1^\infty$ into a non-terminal node,
and we have
\[\big(n^{(2)}, n^{(2)}_1, n^{(2)}_2, n^{(2)}_3\big) :=(n_1, m_1, m_2, m_3).\]

%
%

After  $j - 1$ steps, the ordered tree $\TT_j$ 
of the $j$th generation is obtained from $\TT_{j-1}$ by
changing one of its terminal nodes $a  \in \TT^\infty_{j-1}$
into a non-terminal node.
Then, we define
the $j$th generation of frequencies by
\[\big(n^{(j)}, n^{(j)}_1, n^{(j)}_2, n^{(j)}_3\big) :=(n_a, n_{a_1}, n_{a_2}, n_{a_3}).\]

\noi
As before, from Definition \ref{DEF:tree4}, we have
\begin{equation} \label{freq}
 n^{(j)} = n^{(j)}_1 - n^{(j)}_2 + n^{(j)}_3, \quad n^{(j)}_2\ne n^{(j)}_1, n^{(j)}_3.
\end{equation}

\noi
Also, we have $n^{(j)} = n^{(m)}_k (=n_a)$ for some $m \in \{1, \dots, j-1\}$
and $k \in \{1, 2, 3\}$,
since this corresponds to the frequency of some terminal node in $\TT_{j-1}$.

In the following, we pictorially present an example of an ordered tree $\TT \in \mathfrak{T}(4)$ with ${\bf n} \in \mathfrak{N}(\TT)$:
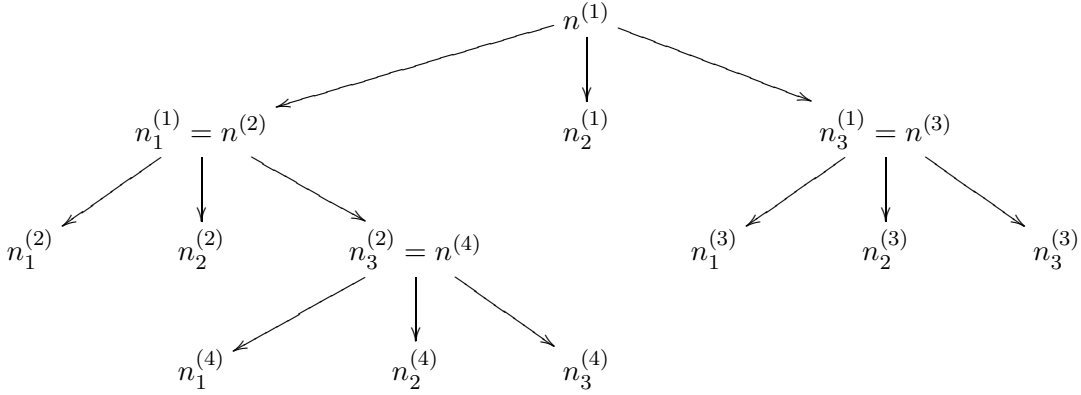
\begin{equation*} 
\xymatrix{
 & & &n^{(1)} \ar[dll] \ar[d] \ar[drr]\\
 &  n^{(1)}_1 =n^{(2)} \ar[dl] \ar[d]  \ar[dr] & & n^{(1)}_2& &  n^{(1)}_3 = n^{(3)} \ar[dl] \ar[d]  \ar[dr] &\\
n^{(2)}_1&n^{(2)}_2&n^{(2)}_3 = n^{(4)}\ar[dl] \ar[d]  \ar[dr] &&n^{(3)}_1&n^{(3)}_2&n^{(3)}_3\\
&n^{(4)}_1&n^{(4)}_2&n^{(4)}_3&&&}
\end{equation*}

\noi
Here, we have ornamented the nodes with the values of ${\bf n} 
= \{n_a\}_{a\in\TT}
\in \mathfrak{N}(\TT)$,
specifying the generations of frequencies as discussed above.

\medskip

We use $\mu_j$  to denote the corresponding phase factor introduced at the $j$th generation.
Namely, we have
\begin{align}
\mu_j & = \mu_j \big(n^{(j)}, n^{(j)}_1, n^{(j)}_2, n^{(j)}_3\big)
:= \big(n^{(j)}\big)^2 - \big(n_1^{(j)}\big)^2 + \big(n_2^{(j)}\big)^2- \big(n_3^{(j)}\big)^2 \notag \\
& = 2\big(n_2^{(j)} - n_1^{(j)}\big) \big(n_2^{(j)} - n_3^{(j)}\big)
= 2\big(n^{(j)} - n_1^{(j)}\big) \big(n^{(j)} - n_3^{(j)}\big), \label{mu}
\end{align}

\noi
where the last two equalities hold thanks to \eqref{freq}.

Lastly, for a fixed ordered tree $\TT$, we denote by $B_j = B_j(\TT)$
the set of all possible frequencies in the $j$th generation.

\subsection{Example: second and third generations}
Using these notations, we can rewrite $\N_3$ in \eqref{N3} as
\begin{align} \label{N^2}
\N^{(2)} (n): = (\N_3)_n = 
\sum_{\TT_2 \in \mathfrak{T}(2)}\sum_{\substack{{\bf n} \in \mathfrak{N}(\TT_2)\\{\bf n}_r = n}} 
\frac{e^{- i( \mu_1 + \mu_2 )t } }{\mu_1}
\prod_{a \in \TT^\infty_2} v_{n_{a}}. 
\end{align}

\noi
%
%
%
Here, we included the contribution of a similar term arising from  $\N_{222}$ in \eqref{N22},
i.e. when the time derivative falls on the second factor $v_{n_2}$.\footnote{As before, 
we only keep track of the absolute values of coefficients
in the following.
We may also drop the minus signs and the complex number $i$.}
Strictly speaking, the new phase factor may be $\mu_1 - \mu_2$
when the time derivative falls on the complex conjugate.
However, for our analysis, it makes no difference and hence
we simply write it as $\mu_1 + \mu_2$.
The same comments apply in the following.
Also, recall that the set of frequencies are restricted onto $A_N^c$ defined in \eqref{A_N}.
See Definition \ref{DEF:tree4} (iii).
In the following, similar restrictions on $\mu_j$ appear, 
but we suppress such restrictions for simplicity of notations, 
when it is clear from the context.

\medskip

Next, we divide the Fourier space into 
\begin{equation} \label{C1}
C_1 = \big\{ |\mu_1 + \mu_2| \lesssim 5^3 |\mu_1|^{1-\frac{1}{100}}\big\} 
\end{equation}

\noi
and its complement $C_1^c$.\footnote{Clearly, the number $5^3$ in \eqref{C1} 
does not make any difference at this point.
However, we insert it to match with \eqref{CJ}.
See also  \eqref{C2} and \eqref{C3}.}
Then, write
\begin{equation} \label{N^2_1}
\N^{(2)} = \N^{(2)}_1 + \N^{(2)}_2,
\end{equation}

\noi
where $\N^{(2)}_1$ is the restriction of $\N^{(2)}$
onto $C_1$
and
$\N^{(2)}_2 := \N^{(2)} - \N^{(2)}_1$.

\begin{lemma}\label{LEM:N^2_1}
Let $\N^{(2)}_1$ be as in \eqref{N^2_1}.
Then, we have
\begin{align} \label{N^2-1}
\| \N^{(2)}_1(v)\|_{L^2} & \lesssim N^{-\frac{1}{200}+} \|v\|_{L^2}^5,\\
\| \N^{(2)}_1(v) - \N^{(2)}_1(w)\|_{L^2} 
& \lesssim 
N^{-\frac{1}{200}+} \big(\|v\|_{L^2}^4 + \|w\|_{L^2}^4\big) \|v - w\|_{L^2}. \label{N^2-2}
\end{align}
\end{lemma}

\begin{proof}
We only prove \eqref{N^2-1} since \eqref{N^2-2} follows in a similar manner.
Since we are on $A_N^c$ (see \eqref{A_N}), we have $|\mu_1| >N$.
Next, we use the divisor counting argument as in the proof of Lemma \ref{LEM:N11}.
It follows from \eqref{divisor}
that for fixed $n$ and  $\mu_1$,
there are at most $o(|\mu_1|^{0+})$ many choices for $n^{(1)}_1$ and $n^{(1)}_3$ on $B_1$
(and hence for $n^{(1)}_2$ from $n^{(1)} = n^{(1)}_1 - n^{(1)}_2 + n^{(1)}_3$).
Similarly, 
for fixed $n^{(2)} = n_1^{(1)}$ and  $\mu_2$,
there are at most $o(|\mu_2|^{0+})$ many choices for $n^{(2)}_1$, $n^{(2)}_2$, and $n^{(2)}_3$ 
on $B_2$.

The main point is to control $|\mu_2|$ in terms of $|\mu_1|$.
From \eqref{C1}, we have $|\mu_2|\sim|\mu_1|$.
Moreover, for fixed $|\mu_1|$, there are at most $O(|\mu_1|^{1-\frac{1}{100}})$ many 
choices for $\mu_2$.
Hence,  by Cauchy-Schwarz inequality, we have
\begin{align*}
\| \N^{(2)}_1(v)\|_{L^2}
& \lesssim 
\sum_{\TT_2 \in \mathfrak{T}(2)}
\bigg(\sum_n \bigg|
\sum_{|\mu|> N } 
\sum_{\substack{{\bf n} \in \mathfrak{N}(\TT_2)\\{\bf n}_r = n\\{\mu_1 = \mu}}} 
\frac{1}{|\mu_1|} \prod_{a \in \TT^\infty_2} v_{n_a} 
\bigg|^2\bigg)^\frac{1}{2}\\
& \lesssim \bigg\{ \sum_n 
\bigg( 
\sum_{|\mu| >N } \frac{1}{|\mu|^2}\, |\mu|^{1-\frac{1}{100}+}\bigg)
 \bigg(\sum_{\substack{{\bf n} \in \mathfrak{N}(\TT_2)\\{\bf n}_r = n}} 
\prod_{a \in \TT^\infty_2} |v_{n_a}|^2
\bigg)\bigg\}^\frac{1}{2}\\
& \leq N^{-\frac{1}{200}+} \|v\|_{L^2}^5. \qedhere
\end{align*}
\end{proof}

%
%
%

Next, we apply (the second step of) Poincar\'e-Dulac normal form reduction
to $\N^{(2)}_2$. 
Note that we have 
\begin{equation} \label{mu2}
|\mu_1 + \mu_2|  \gg 5^3 |\mu_1|^{1-\frac{1}{100}} > 5^3 N^{1-\frac{1}{100}}
\end{equation}

\noi on the support 
of $\N^{(2)}_2$, i.e. on $C_1^c$.
After differentiation by parts, we obtain
\begin{align} \label{N^3}
\N^{(2)}_2 (n)
& = \dt \bigg[\sum_{\TT_2 \in \mathfrak{T}(2)}\sum_{\substack{{\bf n} \in \mathfrak{N}(\TT_2)\\{\bf n}_r = n}} 
\frac{e^{- i( \mu_1 + \mu_2 )t } }{\mu_1(\mu_1+\mu_2)}
\prod_{a \in \TT^\infty_2} v_{n_{a}} \bigg]\notag \\
& \hphantom{X} -  
\sum_{\TT_2 \in \mathfrak{T}(2)}
\sum_{\substack{{\bf n} \in \mathfrak{N}(\TT_2)\\{\bf n}_r = n} }
\frac{e^{- i( \mu_1 + \mu_2)t } }{\mu_1(\mu_1+\mu_2)}
\, 
\dt \bigg(\prod_{a \in \TT^\infty_2 } v_{n_{a}}\bigg) \notag \\
& = \dt \bigg[\sum_{\TT_2 \in \mathfrak{T}(2)}\sum_{\substack{{\bf n} \in \mathfrak{N}(\TT_2)\\{\bf n}_r = n}} 
\frac{e^{- i( \mu_1 + \mu_2 )t } }{\mu_1(\mu_1+\mu_2)}
\prod_{a \in \TT^\infty_2} v_{n_{a}} \bigg]\notag \\
& \hphantom{X} -  
\sum_{\TT_2 \in \mathfrak{T}(2)}
\sum_{b \in\TT^\infty_2} 
\sum_{\substack{{\bf n} \in \mathfrak{N}(\TT_2)\\{\bf n}_r = n} }
\frac{e^{- i( \mu_1 + \mu_2)t } }{\mu_1(\mu_1+\mu_2)}
\, \dt v_{n_b}
\prod_{a \in \TT^\infty_2 \setminus \{b\}} v_{n_{a}} \notag \\
& = \dt \bigg[\sum_{\TT_2 \in \mathfrak{T}(2)}\sum_{\substack{{\bf n} \in \mathfrak{N}(\TT_2)\\{\bf n}_r = n}} 
\frac{e^{- i( \mu_1 + \mu_2 )t } }{\mu_1(\mu_1+\mu_2)}
\prod_{a \in \TT^\infty_2} v_{n_{a}} \bigg]\notag \\
& \hphantom{X} -  
\sum_{\TT_2 \in \mathfrak{T}(2)}
\sum_{b \in\TT^\infty_2} 
\sum_{\substack{{\bf n} \in \mathfrak{N}(\TT_2)\\{\bf n}_r = n} }
\frac{e^{- i( \mu_1 + \mu_2)t } }{\mu_1(\mu_1+\mu_2)}
\, (\RR_1-\RR_2)_{n_b}
\prod_{a \in \TT^\infty_2 \setminus \{b\}} v_{n_{a}} \notag \\
& \hphantom{X} - \sum_{\TT_3 \in \mathfrak{T}(3)}\sum_{\substack{{\bf n} \in \mathfrak{N}(\TT_3)\\{\bf n}_r = n}} 
\frac{e^{- i( \mu_1 + \mu_2 +\mu_3)t } }{\mu_1(\mu_1+\mu_2)}
\,\prod_{a \in \TT^\infty_3} v_{n_{a}} \notag\\
& =: \dt \N^{(3)}_0(n) + \N^{(3)}_r(n) + \N^{(3)}(n).
\end{align}

\noi
In the third equality, we used \eqref{NLS4}
and replaced $\dt v_{n_b}$ by 
the resonant part $(\RR_1 -\RR_2) (n_b)$ and the non-resonant part $\N_1(n_b)$.
As for the contribution from the non-resonant part, we replace 
the frequency $n_b$ by $n_{b_1}, n_{b_2}$, and $n_{b_3}$
such that $n_b = n_{b_1} - n_{b_2} + n_{b_3}$
and $n_{b_2} \ne n_{b_1}, n_{b_3}$,
which corresponds to extending the tree $\TT_2 \in \mathfrak{T}(2)$
(and ${\bf n }\in \mathfrak{N}(\TT_2)$)
to $\TT_3 \in \mathfrak{T}(3)$
(and to  ${\bf n }\in \mathfrak{N}(\TT_3)$, respectively)
by replacing the terminal node $b \in \TT^\infty_2$
into a non-terminal node with three children $b_1, b_2,$ and $b_3$.

\medskip

First, let us estimate the easier terms $\N^{(3)}_0$ and $\N^{(3)}_r$.

\begin{lemma}\label{LEM:N^3_0}
Let $\N^{(3)}_0$ be as in \eqref{N^3}.
Then, we have
\begin{align} 
\| \N^{(3)}_0(v)\|_{L^2} & \lesssim N^{-1 + \frac{1}{200}+} \|v\|_{L^2}^5, \label{N^3_0-1}\\
\| \N^{(3)}_0(v) - \N^{(3)}_0(w)\|_{L^2} 
& \lesssim 
N^{-1 + \frac{1}{200}+} \big(\|v\|_{L^2}^4 + \|w\|_{L^2}^4\big) \|v - w\|_{L^2}. \label{N^3_0-2}
\end{align}
\end{lemma}

\begin{proof}
We only prove \eqref{N^3_0-1} since \eqref{N^3_0-2} follows in a similar manner.
As in the proof of Lemma \ref{LEM:N^2_1}, 
it follows from \eqref{divisor}
that for fixed $n$ and  $\mu_1$,
there are at most $o(|\mu_1|^{0+})$ many choices for $n^{(1)}_1$, $n^{(1)}_2$, and $n^{(1)}_3$ on $B_1$.
Similarly, 
for fixed $n^{(2)} = n_1^{(1)}$ and  $\mu_2$,
there are at most $o(|\mu_2|^{0+})$ many choices for $n^{(2)}_1$, $n^{(2)}_2$, and $n^{(2)}_3$ 
on $B_2$.

With $\wt{\mu}_2 = \mu_1 + \mu_2$, 
we have $|\mu_2| \lesssim \max (|\mu_1|, |\wt{\mu}_2|)$.
Then, by Cauchy-Schwarz inequality, we have
\begin{align*}
\| \N^{(3)}_0(v)\|_{L^2}
& \lesssim \sum_{\TT_2 \in \mathfrak{T}(2)}\bigg\{ \sum_n 
\bigg(\sum_{\substack{|\mu_1| >N\\ |\wt{\mu}_2| >5^3 N^{1-\frac{1}{100}}}}
\frac{1}{|\mu_1|^{2}|\wt{\mu}_2|^{2}} \, |\mu_1|^{0+}|\mu_2|^{0+}
\bigg)\\
 & \hphantom{XXXXXXXXX} \times 
  \bigg(\sum_{\substack{{\bf n} \in \mathfrak{N}(\TT_2)\\{\bf n}_r = n}} 
\prod_{a \in \TT^\infty_2} |v_{n_a}|^2
\bigg)\bigg\}^\frac{1}{2}\\
& \lesssim N^{-1 + \frac{1}{200}+} \|v\|_{L^2}^5. \qedhere
\end{align*}
\end{proof}

\begin{lemma}\label{LEM:N^3_r}
Let $\N^{(3)}_r$ be as in \eqref{N^3}.
Then, we have
\begin{align} 
\| \N^{(3)}_r(v)\|_{L^2} & \lesssim N^{-1 + \frac{1}{200}+} \|v\|_{L^2}^7, \label{N^3_r-1}\\
\| \N^{(3)}_r(v) - \N^{(3)}_r(w)\|_{L^2} 
& \lesssim 
N^{-1 + \frac{1}{200}+} \big(\|v\|_{L^2}^6 + \|w\|_{L^2}^6\big) \|v - w\|_{L^2}. \label{N^3_r-2}
\end{align}
\end{lemma}

\begin{proof}
This lemma  follows from Lemmata \ref{LEM:N^3_0} and \ref{LEM:R1}.
Note that, given $\TT_2 \in \mathfrak{T}(2)$, we have $\#\{b: b \in \TT_2\} = 5$.
\end{proof}

Now, we treat $\N^{(3)}$.
As before, we write
\begin{equation} \label{N^3_1}
\N^{(3)} = \N^{(3)}_1 + \N^{(3)}_2,
\end{equation}

\noi
where $\N^{(3)}_1$ is the restriction of $\N^{(3)}$
onto 
\begin{equation} \label{C2}
C_2 = \big\{ |\wt{\mu}_3| \lesssim 7^3|\wt{\mu}_2|^{1-\frac{1}{100}}\big\} 
\cup \big\{ |\wt{\mu}_3| \lesssim 7^3 |\mu_1|^{1-\frac{1}{100}}\big\}, 
\end{equation}

\noi
where $\wt{\mu}_2 := \mu_1 + \mu_2$ and $\wt{\mu}_3 := \mu_1 + \mu_2 + \mu_3$,
and
$\N^{(3)}_2 := \N^{(3)} - \N^{(3)}_1$.

\begin{lemma}\label{LEM:N^3_1}
Let $\N^{(3)}_1$ be as in \eqref{N^3_1}.
Then, we have
\begin{align} \label{N^3-1}
\| \N^{(3)}_1(v)\|_{L^2} & \lesssim N^{-\frac{1}{2}+} \|v\|_{L^2}^7,\\
\| \N^{(3)}_1(v) - \N^{(3)}_1(w)\|_{L^2} 
& \lesssim 
N^{-\frac{1}{2}+} \big(\|v\|_{L^2}^6 + \|w\|_{L^2}^6\big) \|v - w\|_{L^2}. \label{N^3-2}
\end{align}
\end{lemma}

\begin{proof}
We only prove \eqref{N^3-1} since \eqref{N^3-2} follows in a similar manner.
The proof is very similar to that of Lemma \ref{LEM:N^2_1},
i.e. we use the divisor counting argument.
It follows from \eqref{divisor}
that for fixed $n$ and  $\mu_1$,
there are at most $o(|\mu_1|^{0+})$ many choices for $n^{(1)}_1, n^{(1)}_2,$ and $ n^{(1)}_3$ on $B_1$.
For fixed  $n^{(2)} = n_1^{(1)}$ and $\mu_2$,
there are at most $o(|\mu_2|^{0+})$ many choices for $n^{(2)}_1$, $n^{(2)}_2$, and $n^{(2)}_3$ 
on $B_2$.
Moreover, for fixed  $n^{(3)} = n_1^{(2)}$ and $\mu_3$,
there are at most $o(|\mu_3|^{0+})$ many choices for $n^{(3)}_1$, $n^{(3)}_2$, and $n^{(3)}_3$ 
on $B_3$.

First, we assume that $|\wt{\mu}_3| \lesssim |\wt{\mu}_2|^{1-\frac{1}{100}}$
holds
in \eqref{C2}.
Then, we have $|\mu_3|\sim|\wt{\mu}_2|$ since $\wt{\mu}_3 = \wt{\mu}_2 + \mu_3$.
Moreover, for fixed $|\wt{\mu}_2|$, namely for fixed $\mu_1$ and $\mu_2$,
there are at most $O(|\wt{\mu}_2|^{1-\frac{1}{100}})$ many 
choices for $\wt{\mu}_3$,  and hence for $\mu_3 = \wt{\mu}_3 - \wt{\mu}_2$.
Also, we have
$|\mu_2| \lesssim \max(|\mu_1|, |\wt{\mu}_2|)$
and \eqref{mu2}.
Then, by Cauchy-Schwarz inequality, we have
\begin{align}
\| \N^{(3)}_1(v)\|_{L^2}
& \lesssim \sum_{\TT_3 \in \mathfrak{T}(3)}
\bigg(\sum_n \bigg|
 \sum_{\substack{{\bf n} \in \mathfrak{N}(\TT_3)\\{\bf n}_r = n}} 
 \frac{1}{|\mu_1||\wt{\mu}_2|}
\prod_{a \in \TT^\infty_3} v_{n_{a}} 
\bigg|^2\bigg)^\frac{1}{2}   \notag \\
& \lesssim \bigg\{ \sum_n 
\bigg(\sum_{\substack{|\mu_1|>N\\ |\wt{\mu}_2| >5^3N^{1-\frac{1}{100}}}} 
\frac{1}{|\mu_1|^{2}|\wt{\mu}_2|^{2}} |\mu_1|^{0+}|\mu_2|^{0+}|\mu_3|^{0+}
|\wt{\mu}_2|^{1-\frac{1}{100}}\bigg) \notag \\
 & \hphantom{XXXXXX} \times \bigg( \sum_{\substack{{\bf n} \in \mathfrak{N}(\TT_3)\\{\bf n}_r = n}} 
\prod_{a \in \TT^\infty_3} |v_{n_{a}}|^2 
\bigg)\bigg\}^\frac{1}{2} \notag \\
& \lesssim \bigg\{ \sum_n 
\bigg(\sum_{\substack{|\mu_1|>N\\ |\wt{\mu}_2| >5^3N^{1-\frac{1}{100}}}} 
\frac{1}{|\mu_1|^{2-}|\wt{\mu}_2|^{1+\frac{1}{100}-}}\bigg) 
  \bigg(\sum_{\substack{{\bf n} \in \mathfrak{N}(\TT_3)\\{\bf n}_r = n}} 
\prod_{a \in \TT^\infty_3} |v_{n_{a}}|^2 
\bigg)\bigg\}^\frac{1}{2} \notag\\
& \lesssim N^{-\frac{1}{2}} \|v\|_{L^2}^7.  \label{N^3-3}
\end{align}

\noi
If $|\wt{\mu}_3| \lesssim |\mu_1|^{1-\frac{1}{100}}$ holds in \eqref{C2},
then, for fixed $\mu_1$ and $\mu_2$, 
there are at most $O(|\mu_1|^{1-\frac{1}{100}})$ many 
choices for $\wt{\mu}_3$, and hence for $\mu_3$.
By repeating the same computation, we obtain 
$|\mu_1|^{-1-\frac{1}{100}+}|\wt{\mu}_2|^{-2+}$
in \eqref{N^3-3}, yielding \eqref{N^3-1} with $N^{-\frac{1}{2}+}$.
\end{proof}

Next, we apply (the third step of) Poincar\'e-Dulac normal form reduction
to $\N^{(3)}_2$. 
Note that we have 
\begin{equation} \label{mu3}
|\wt{\mu}_3| = |\mu_1 + \mu_2+\mu_3|  \gg 7^3 |\mu_1|^{1-\frac{1}{100}} > 7^3 N^{1-\frac{1}{100}}
\end{equation}

\noi on the support 
of $\N^{(3)}_2$, i.e. on $C_2^c$.
After differentiation by parts, we obtain
\begin{align} \label{N^4}
\N^{(3)}_2 (n)
& = \dt \bigg[-\sum_{\TT_3 \in \mathfrak{T}(3)}\sum_{\substack{{\bf n} \in \mathfrak{N}(\TT_3)\\{\bf n}_r = n}} 
\frac{e^{- i( \mu_1 + \mu_2 +\mu_3)t } }{\mu_1(\mu_1+\mu_2)( \mu_1 + \mu_2 +\mu_3)}
\, \prod_{a \in \TT^\infty_3} v_{n_{a}}
\bigg]\notag \\
& \hphantom{X} + 
\sum_{\TT_3 \in \mathfrak{T}(3)}
\sum_{b \in\TT^\infty_3} 
\sum_{\substack{{\bf n} \in \mathfrak{N}(\TT_3)\\{\bf n}_r = n} }
\frac{e^{- i( \mu_1 + \mu_2+\mu_3)t } }{\mu_1(\mu_1+\mu_2)( \mu_1 + \mu_2 +\mu_3)}
\, (\RR_1-\RR_2)_{n_b}
\prod_{a \in \TT^\infty_3 \setminus \{b\}} v_{n_{a}}  \notag \\
& \hphantom{X} + 
\sum_{\TT_4 \in \mathfrak{T}(4)}\sum_{\substack{{\bf n} \in \mathfrak{N}(\TT_4)\\{\bf n}_r = n}} 
\frac{e^{- i( \mu_1 + \mu_2 +\mu_3+\mu_4)t } }{\mu_1(\mu_1+\mu_2)( \mu_1 + \mu_2 +\mu_3)}
\,\prod_{a \in \TT^\infty_4} v_{n_{a}}  \notag\\
& =: \dt \N^{(4)}_0 (n)+ \N^{(4)}_r (n)+ \N^{(4)}(n).
\end{align}

\noi
We can clearly estimate $\N^{(4)}_0$ and $ \N^{(4)}_r$,
with $|\mu_1| >N$, \eqref{mu2}, and \eqref{mu3}, just as in Lemmata \ref{LEM:N^3_0} and \ref{LEM:N^3_r}.
As for $\N^{(4)}$,  
we can write 
\[\N^{(4)} = \N^{(4)}_1 + \N^{(4)}_2\]

\noi
as the restrictions onto 
\begin{equation} \label{C3}
C_3 = \big\{ |\wt{\mu}_4| \lesssim 9^3  |\wt{\mu}_3|^{1-\frac{1}{100}}\big\} 
\cup \big\{ |\wt{\mu}_4| \lesssim 9^3  |\mu_1|^{1-\frac{1}{100}}\big\}, 
\end{equation}

\noi
where $\wt{\mu}_4 :=  \mu_1 + \mu_2 + \mu_3+\mu_4$,
and its complement $C_3^c$, respectively.
Then, $\N^{(4)}_1$ can be estimated as in Lemma \ref{LEM:N^3_1}
and we can apply (the fourth step of)
Poincar\'e-Dulac normal form reduction
to $\N^{(4)}_2$. 
In this way, we iterate Poincar\'e-Dulac normal form reductions.


\subsection{General step: $J$th generation}

%

After the $J$\,th step, we have
\begin{align} \label{N^J+1}
\N^{(J)}_2 (n)
& = \dt \bigg[\mp 
\sum_{\TT_J \in \mathfrak{T}(J)}\sum_{\substack{{\bf n} \in \mathfrak{N}(\TT_J)\\{\bf n}_r = n}} 
\frac{e^{- i \wt{\mu}_Jt } }{\ft{\mu}_J}
\, \prod_{a \in \TT^\infty_J} v_{n_{a}}
\bigg]\notag \\
& \hphantom{X} 
\pm 
 \sum_{\TT_{J} \in \mathfrak{T}(J)}
\sum_{b \in\TT^\infty_J} 
\sum_{\substack{{\bf n} \in \mathfrak{N}(\TT_J)\\{\bf n}_r = n} }
\frac{e^{- i \wt{\mu}_Jt } }{\ft{\mu}_J}
\, (\RR_1-\RR_2)_{n_b}
\prod_{a \in \TT^\infty_J \setminus \{b\}} v_{n_{a}}  \notag \\
& \hphantom{X} 
\pm 
\sum_{\TT_{J+1} \in \mathfrak{T}(J+1)}\sum_{\substack{{\bf n} \in \mathfrak{N}(\TT_{J+1})\\{\bf n}_r = n}} 
\frac{e^{- i \wt{\mu}_{J+1}t } }{\ft{\mu}_J}
\,\prod_{a \in \TT^\infty_{J+1}} v_{n_{a}} \notag\\
& =: \dt \N^{(J+1)}_0 (n)+ \N^{(J+1)}_r(n) + \N^{(J+1)}(n),
\end{align}

\noi
where
$\wt{\mu}_J$ and $\ft{\mu}_J$
are given by
\begin{align*}
 \wt{\mu}_J := \sum_{j = 1}^J \mu_j,
\quad \text{and} \quad \ft{\mu}_J  := \prod_{j = 1}^J \wt{\mu}_j.
\end{align*}


\medskip

Keep in mind that $|\mu_1|>N$ and 
\begin{equation} \label{muj}
|\wt{\mu}_j|  \gg 
(2j+1)^3\max ( |\wt{\mu}_{j-1}|^{1-\frac{1}{100}},
|\mu_1|^{1-\frac{1}{100}}) >
(2j+1)^3N^{1-\frac{1}{100}}, 
\end{equation}

\noi
for $j = 2, \dots, J.$
First, we estimate $\N^{(J+1)}_0$ and  $\N^{(J+1)}_r$.

\begin{lemma}\label{LEM:N^J+1_0}
Let $\N^{(J+1)}_0$ be as in \eqref{N^J+1}.
Then, 
we have\footnote{The implicit constants are independent of $J$.
The same comment applies to Lemmata \ref{LEM:N^J+1_r} and \ref{LEM:N^J+1_1}.}
\begin{align} 
\| \N^{(J+1)}_0(v)\|_{L^2} & \lesssim N^{-\frac{J}{2}+\frac{(J-1)}{200}+}
 \|v\|_{L^2}^{2J+1}, \label{N^J+1_0-1}\\
\| \N^{(J+1)}_0(v) - \N^{(J+1)}_0(w)\|_{L^2} 
& \lesssim 
N^{-\frac{J}{2}+\frac{(J-1)}{200}+}
 \big(\|v\|_{L^2}^{2J} + \|w\|_{L^2}^{2J}\big) \|v - w\|_{L^2}. \label{N^J+1_0-2}
\end{align}
\end{lemma}

\begin{proof}
We only prove \eqref{N^J+1_0-1} since \eqref{N^J+1_0-2} follows in a similar manner.
Note that there is an extra factor $\sim J$ 
when we estimate the difference in \eqref{N^J+1_0-2}
since
$|a^{2J+1} - b^{2J+1}| \lesssim     
\big(\sum_{j = 1}^{2J+1} a^{2J+1-j}b^{j-1} \big) |a - b |$
has $O(J)$ many terms.
However, this does not cause a problem since the constant we obtain decays 
like a fractional power of a factorial in $J$ (as we see below in \eqref{N^J+1_0-3}.)
The same comment applies to Lemmata \ref{LEM:N^J+1_r}
and \ref{LEM:N^J+1_1}.

As in the proof of Lemma \ref{LEM:N^3_0},
for fixed  $n^{(j)}$ and $\mu_j$,
there are at most $o(|\mu_j|^{0+})$ many choices for $n^{(j)}_1$, $n^{(j)}_2$, and $n^{(j)}_3$.
Also, note that $\mu_j$ is determined by $\wt{\mu}_1, \dots, \wt{\mu}_j$
and 
\begin{equation} \label{mujj}
|\mu_j| \lesssim \max(|\wt{\mu}_{j-1}|, |\wt{\mu}_j|).
\end{equation}

\noi
since $\mu_j = \wt{\mu}_{j}-\wt{\mu}_{j-1}$.
Then, by Cauchy-Schwarz inequality, we have
\begin{align}
\| \N^{(J+1)}_0(v)\|_{L^2}
& \lesssim  
\sum_{\TT_J \in \mathfrak{T}(J)}
\bigg\{ \sum_n 
\bigg(\sum_{\substack{|\mu_1|>N \\ |\wt{\mu}_j|>(2j+1)^3N^{1-\frac{1}{100}}\\ j = 2, \dots, J} }
\prod_{k = 1}^J \frac{1}{|\wt{\mu}_k|^2}
\, |\mu_k|^{0+}\bigg) \notag \\
& \hphantom{XXXXXX} \times 
\bigg(\sum_{\substack{{\bf n} \in \mathfrak{N}(\TT_J)\\{\bf n}_r = n}} 
\prod_{a \in \TT^\infty_J} |v_{n_a}|^2
\bigg)\bigg\}^\frac{1}{2}  \notag \\
& \lesssim \frac{c_J}{\prod_{j = 2}^J(2j+1)^{\frac{3}{2}-}}N^{-\frac{J}{2}+\frac{(J-1)}{200}+} \|v\|_{L^2}^{2J+1}
\label{N^J+1_0-3} \\
&\lesssim N^{-\frac{J}{2}+\frac{(J-1)}{200}+} \|v\|_{L^2}^{2J+1}, \notag
\end{align}

\noi
where $c_J = |\mathfrak{T}(J)| $ is defined in \eqref{cj1}.
\end{proof}

\medskip

\begin{lemma}\label{LEM:N^J+1_r}
Let $\N^{(J+1)}_r$ be as in \eqref{N^J+1}.
Then, 
we have
\begin{align} 
\| \N^{(J+1)}_r(v)\|_{L^2} & \lesssim N^{-\frac{J}{2}+\frac{(J-1)}{200}+} 
\|v\|_{L^2}^{2J+3}, \label{N^J+1_r-1}\\
\| \N^{(J+1)}_r(v) - \N^{(J+1)}_r(w)\|_{L^2} 
& \lesssim 
N^{-\frac{J}{2}+\frac{(J-1)}{200}+} 
\big(\|v\|_{L^2}^{2J+2} + \|w\|_{L^2}^{2J+2}\big) \|v - w\|_{L^2}. \label{N^J+1_r-2}
\end{align}
\end{lemma}

\begin{proof}
This lemma  follows from Lemmata \ref{LEM:N^J+1_0} and \ref{LEM:R1}.
Note that, given $\TT_J \in \mathfrak{T}(J)$, we have $\#\{b: b \in \TT^\infty_J\} = 2J+1$.
This extra factor $2J+1$ does not cause a problem thanks to the fast decaying constant 
in \eqref{N^J+1_0-3}.
\end{proof}

Finally, we treat $\N^{(J+1)}$.
As before, we write
\begin{equation} \label{N^J+1_1}
\N^{(J+1)} = \N^{(J+1)}_1 + \N^{(J+1)}_2,
\end{equation}

\noi
where $\N^{(J+1)}_1$ is the restriction of $\N^{(J+1)}$
onto 
\begin{equation} \label{CJ}
C_J = \big\{ |\wt{\mu}_{J+1}| \lesssim (2J+3)^3|\wt{\mu}_J|^{1-\frac{1}{100}}\big\} 
\cup \big\{ |\wt{\mu}_{J+1}| \lesssim (2J+3)^3|\mu_1|^{1-\frac{1}{100}}\big\} 
\end{equation}

\noi
and
$\N^{(J+1)}_2 := \N^{(J+1)} - \N^{(J+1)}_1$.
We estimate the first term $\N^{(J+1)}_1$
in the following lemma,
while we apply Poincar\'e-Dulac normal form reduction 
once again to the second term $\N^{(J+1)}_2$ 
as in \eqref{N^J+1}.

\begin{lemma}\label{LEM:N^J+1_1}
Let $\N^{(J+1)}_1$ be as in \eqref{N^J+1_1}.
Then, 
we have
\begin{align} \label{N^J+1-1}
\| \N^{(J+1)}_1(v)\|_{L^2} & \lesssim N^{-\frac{J-1}{2}+\frac{(J-2)}{200}+}\|v\|_{L^2}^{2J+3},\\
\| \N^{(J+1)}_1(v) - \N^{(J+1)}_1(w)\|_{L^2} 
& \lesssim 
N^{-\frac{J-1}{2}+\frac{(J-2)}{200}+} \big(\|v\|_{L^2}^{2J+2} + \|w\|_{L^2}^{2J+2}\big) \|v - w\|_{L^2}. \label{N^J+1-2}
\end{align}
\end{lemma}

\begin{proof}
We only prove \eqref{N^J+1-1} since \eqref{N^J+1-2} follows in a similar manner.
As before, we use the divisor counting argument.
For fixed $n^{(j)}$ and $\mu_j$,
there are at most $o(|\mu_j|^{0+})$ many choices for $n^{(j)}_1$, $n^{(j)}_2$, and $n^{(j)}_3$.
Also, note that $\mu_j$ is determined by $\wt{\mu}_1, \dots, \wt{\mu}_j$

First, we assume that 
$ |\wt{\mu}_{J+1}| = |\wt{\mu}_J+\mu_{J+1}| \lesssim (2J+3)^3|\wt{\mu}_J|^{1-\frac{1}{100}}$
holds in \eqref{CJ}.
Then, we have
$|\mu_{J+1}| \lesssim |\wt{\mu}_J|$.
Also, for fixed $\wt{\mu}_J$,
there are at most $o(|\wt{\mu}_J|^{1-\frac{1}{100}})$ many choices\footnote{Strictly speaking,
there are at most $o((2J+3)^3|\wt{\mu}_J|^{1-\frac{1}{100}})$ choices.
However, we drop $(2J+3)^3$ in view of fast decay of coefficients in $J$.
See \eqref{N^J+1_0-3}.
The same comment applies in the following.}
for $\wt{\mu}_{J+1}$
and hence for $\mu_{J+1} = \wt{\mu}_{J+1} - \wt{\mu}_J$.
Then, by Cauchy-Schwarz inequality with \eqref{muj} and \eqref{mujj}, we have
\begin{align}
\| \N^{(J+1)}_1(v)\|_{L^2}
&  \lesssim \sum_{\TT_{J+1} \in \mathfrak{T}(J+1)}
\bigg\{ \sum_n 
\bigg(\sum_{\substack{|\mu_1|>N \\ |\wt{\mu}_j|>(2j+1)^3N^{1-\frac{1}{100}}\\ j = 2, \dots, J} }
|\wt{\mu}_J|^{1-\frac{1}{100}+}
\prod_{k = 1}^J \frac{1}{|\wt{\mu}_k|^2}
\, |\mu_k|^{0+} \bigg)  \notag\\
& \hphantom{XXXXXX}
\times \bigg(\sum_{\substack{{\bf n} \in \mathfrak{N}(\TT_{J+1})\\{\bf n}_r = n}} 
\prod_{a \in \TT^\infty_{J+1}} |v_{n_{a}}|^2 
\bigg)\bigg\}^\frac{1}{2} \notag\\
& \lesssim N^{-\frac{J-1}{2}+\frac{J-2}{200}-\frac{1}{200}+} \|v\|_{L^2}^{2J+3}
 \leq N^{-\frac{J-1}{2}+\frac{J-2}{200}+} \|v\|_{L^2}^{2J+3}.  \label{N^J+1-3}
\end{align}

\noi
by crudely estimating in $N$.

If $|\wt{\mu}_{J+1}| \lesssim (2J+3)^3|\mu_1|^{1-\frac{1}{100}}$ holds in \eqref{CJ},
then, for fixed $\mu_j$, $j = 1, \dots, J$, 
there are at most $O(|\mu_1|^{1-\frac{1}{100}})$ many 
choices for $\mu_{J+1}$.
By repeating the same computation, we obtain 
\[|\mu_1|^{1-\frac{1}{100}}
\prod_{k = 1}^J \frac{1}{|\wt{\mu}_k|^2}
\, |\mu_k|^{0+} 
\]
in \eqref{N^J+1-3}, yielding \eqref{N^J+1-1} with 
$N^{-\frac{J-1}{2}+\frac{J-2}{200}+}$.
\end{proof}

\section{Existence of weak solutions}
\label{SEC:4}

In this section, we put together all the lemmata in the previous sections
and prove Theorems \ref{thm0} in $L^2(\T)$, i.e. for $s = 0$.
The argument for $s > 0$ follows in a similar manner and we omit the details.
By performing an infinite iteration of Poincar\'e-Dulac normal form reductions
described in Sections \ref{SEC:2} and \ref{SEC:3}, we have the following.

First consider a smooth solution $v$  of \eqref{NLS4}
with smooth initial condition $v_0$.
Then, it satisfies the Duhamel formulation:
\begin{align} \label{Duhamel2}
v(t) & = v_0 + i \int_0^t \N_1(v)(t') - \RR_1(v)(t') + \RR_2(t') dt'\notag \\
& = v_0 +i \int_0^t S(-t') \big[S(t')v(t')|S(t')v(t')|^2\big] dt'
\end{align}

\noi
as a smooth function for each $t$.
Then it {\it formally} satisfies\footnote{Once again, we are replacing $\pm 1$
and  $\pm i$
by 1 for simplicity since they play no role in our analysis.} 
\begin{align} \label{41}
\dt v =  \dt \sum_{j = 2}^\infty \N^{(j)}_0(v)
+ \RR_1 +\RR_2
+ \sum_{j = 2}^\infty \N^{(j)}_r(v) + \sum_{j = 1}^\infty \N^{(j)}_1(v).
\end{align}

\noi 
or 
\begin{align} \label{42}
 v(t)  = \G_{v_0}v(t) : = v_0 & +   \sum_{j = 2}^\infty  \N^{(j)}_0(v)(t) 
 - \sum_{j = 2}^\infty \N^{(j)}_0(v_0) \notag \\
& + 
\int_0^t \RR_1 (t')+ \RR_2(t')+
\sum_{j = 2}^\infty \N^{(j)}_r(v)(t') + \sum_{j = 1}^\infty \N^{(j)}_1(v)(t')dt',
\end{align}

\noi
where
$\N_1^{(1)} = \N_{11}$ in \eqref{N1},
$\N^{(2)}_0 = \N_{21}$ in \eqref{N12},
$\N^{(2)} = \N_3$ in \eqref{N^2}, and 
$\N^{(2)}_r = \N_4$ in \eqref{N221}.
At this point, the right hand sides of \eqref{41} and \eqref{42}
are merely formal expressions.
In the following, we show that the series appearing on the right hand side
of \eqref{42} converge absolutely in $C([0, T];L^2)$
for sufficiently small $T>0$
if $v \in C([0, T];L^2)$.

First, define the partial sum operator $\G_{v_0}^{(J)}$ by 
\begin{align}
\G_{v_0}^{(J)}  v(t)  = v_0 +   \sum_{j = 2}^J & \N^{(j)}_0(v)(t) 
 - \sum_{j = 2}^J \N^{(j)}_0(v_0) \notag \\
& + 
\int_0^t \RR_1 (t')+\RR_2 (t')+ 
\sum_{j = 2}^J \N^{(j)}_r(v)(t') + \sum_{j = 1}^J \N^{(j)}_1(v)(t')dt'.
\label{43}
\end{align}

\noi
In the following, we let $C_TL^2 = C([0, T];L^2)$.
By Lemmata \ref{LEM:R1}, \ref{LEM:N11}, 
\ref{LEM:N^J+1_0}, \ref{LEM:N^J+1_r}, and \ref{LEM:N^J+1_1}, we have
\begin{align}
\|\G_{v_0}^{(J)}  v\|_{C_TL^2}  
\leq & \ \|v_0\|_{L^2} +   
C \sum_{j = 2}^J N^{-\frac{j-1}{2} +\frac{j-2}{200}+} 
\big(\|v\|_{C_TL^2}^{2j-1} + \|v_0\|_{L^2}^{2j-1}\big)\notag \\
  & \hphantom{XXX}+ 
CT \Big\{\|v\|_{C_TL^2}^3
+ \sum_{j = 2}^J 
N^{-\frac{j-1}{2} +\frac{j-2}{200}+} \|v\|_{C_TL^2}^{2j+1} \notag\\
 & \hphantom{XXXXXX}+ 
 N^{\frac{1}{2}+}\|v\|_{C_TL^2}^3
 +\sum_{j = 2}^J 
 N^{-\frac{j-2}{2} +\frac{j-3}{200}+}
 \|v\|_{C_TL^2}^{2j+1}\Big\}.\label{44}
\end{align}

\noi
{\it Suppose} that $\|v_0\|_{L^2} \leq R$ and $ \|v\|_{C_TL^2} \leq \wt{R}$
with $\wt{R} \geq R \geq 1$.
Then, we have
\begin{align}
\|\G_{v_0}^{(J)}  v\|_{C_TL^2}  
\leq & \ R +   
C N^{-\frac{1}{2}+}R^3 \sum_{j = 0}^{J-2} 
(N^{-\frac{1}{2}+\frac{1}{200}+} R^2)^j 
+ C N^{-\frac{1}{2}+}\wt{R}^3 \sum_{j = 0}^{J-2} 
(N^{-\frac{1}{2}+\frac{1}{200}+} \wt{R}^2)^j
\notag\\
 & \hphantom{XXX} + 
CT \Big\{(1+N^{\frac{1}{2}+})\wt{R}^3
+ N^{-\frac{1}{2}+}\wt{R}^5
\sum_{j = 0}^{J-2} 
(N^{-\frac{1}{2}+\frac{1}{200}+} \wt{R}^2)^j \notag \\
 & \hphantom{XXXXXX}
 +N^{-\frac{1}{200+}} \wt{R}^5 \sum_{j = 0}^{J-2} 
(N^{-\frac{1}{2}+\frac{1}{200}+} \wt{R}^2)^j
\Big\}.\label{45}
\end{align}

\noi
Now, choose $N = N(\wt{R})$ large such that $N^{-\frac{1}{2}+\frac{1}{200}+} \wt{R}^2 \leq \frac{1}{2}$.
For example, we can simply choose 
\begin{equation}
N^{-\frac{1}{3}} \wt{R}^2 \leq \tfrac{1}{2}
\quad \Longleftrightarrow \quad
N \geq (2\wt{R}^2)^3.
\label{46}
\end{equation}

\noi
Then, the geometric series in \eqref{45} converge
(even for $J = \infty$)
and are bounded by 2.
Thus, we have
\begin{align}
\|\G_{v_0}^{(J)}  v\|_{C_TL^2}  
\leq   R \ +  \ & 
2C N^{-\frac{1}{2}+}R^3
+ 2C N^{-\frac{1}{2}+}\wt{R}^3 \notag \\
&  + 
CT \Big\{(1+N^{\frac{1}{2}+})\wt{R}^2
+ 2 N^{-\frac{1}{2}+}\wt{R}^4
 +2N^{-\frac{1}{200+}} \wt{R}^4
\Big\}\wt{R}.\label{47}
\end{align}

\noi
Next, choose $T >0$ sufficiently small such that 
\begin{equation}
\label{48}
CT \Big\{(1+N^{\frac{1}{2}+})\wt{R}^2
+ 2 N^{-\frac{1}{2}+}\wt{R}^4
 +2N^{-\frac{1}{200+}} \wt{R}^4
\Big\} < \tfrac{1}{10}.
\end{equation}

\noi
From \eqref{46}, we have
$2C N^{-\frac{1}{2}+}\wt{R}^3
\leq C N^{-\frac{1}{6}+}\wt{R}$.
Finally, by further imposing $N$ sufficiently large such that
\begin{equation} \label{49}
C N^{-\frac{1}{6}+} <\tfrac{1}{10},
\end{equation}

\noi
we obtain
\begin{align}
\|\G_{v_0}^{(J)}  v(t)\|_{C_TL^2}  
\leq    R+ \tfrac{1}{10} R \ +  
\tfrac{1}{5} \wt{R}
= \tfrac{11}{10} R \ +  
\tfrac{1}{5} \wt{R}.\label{410}
\end{align}

\noi
We point out that this estimate also holds for $J = \infty$,
and hence $\G_{v_0} = \G_{v_0}^{(\infty)}$ (= right hand side of \eqref{42})
is well-defined.

\medskip

Next, given an initial condition $v_0 \in L^2(\T)$, we construct a solution $v \in C([0, T];L^2)$
in the sense of Definition \ref{DEF:3}.
First, take a sequence $\{v_0^{[m]}\}_{m\in \mathbb{N}}$ 
of smooth functions such that $v_0^{[m]} \to v_0$ in $L^2(\T)$.
(Simply take $v_0^{[m]} := \mathbb{P}_{\leq m} v_0$,
where $\mathbb{P}_{\leq m}$ is the Dirichlet projection onto the frequencies $|n|\leq m$.)
Let $R = \|v_0\|_{L^2} + 1$.
Without loss of generality, assume that $\|v_0^{[m]}\|_{L^2} \leq R$.

In the following, we establish an a priori estimate on smooth solutions
without the $L^2$-conservation
so that the argument can be easily modified for $v_0 \in H^s$, $s>0$.
Let $v^{[m]}$ denote the smooth global-in-time solution of cubic NLS \eqref{NLS4}
with initial condition $v_0^{[m]}$.
First, we use the continuity argument to show that 
$v^{[m]}$ satisfies \eqref{42} on $[0, T]$ with $T = T(R)>0$, 
independent of $m \in \mathbb{N}$.
(As we see later, it suffices to take $T = T(R)>0$
satisfying \eqref{48}.)
Fix $m \in \mathbb{N}$.
Note that $\|v^{[m]}\|_{C_tL^2} = \|v^{[m]}\|_{C([0, t];L^2)}$ is continuous in $t$.
Since $ \|v_0^{[m]}\|_{L^2} \leq R$, 
there exists a time interval $[0, T_1]$ with $T_1 > 0$ such that 
$ \|v^{[m]}\|_{C_{T_1}L^2} \leq 4R$. 
Then, by repeating the previous computation with 
$\wt{R} = 4R $
(and keeping one of the factors as $  \|v\|_{C_{T_1}L^2}$),
we obtain
\begin{align}
 \|v^{[m]}\|_{C_{T_1}L^2} =  \|\G_{v_0^{[m]}}  v^{[m]}\|_{C_{T_1} L^2}  
\leq   \tfrac{11}{10} R \ +  \ & 
\tfrac{1}{5}  \|v^{[m]}\|_{C_{T_1}L^2}\label{411}
\end{align}

\noi
as long as $N$ and $T_1$ satisfy \eqref{46}, \eqref{48}, and \eqref{49}.
This implies that
\begin{equation} \label{412}
 \|v^{[m]}\|_{C_{T_1}L^2} \leq \tfrac{19}{10}R < 2R.
\end{equation}

\noi
Hence, it follows from the continuity in $t$ of $ \|v^{[m]}\|_{C_{t}L^2}$
that there exists $\eps > 0$ such that $ \|v^{[m]}\|_{C_{T_1+\eps}L^2} \leq 4R$.
Then, from \eqref{411} and \eqref{412} with $T_1+\eps$ in place of $T_1$,
we conclude that  
$ \|v^{[m]}\|_{C_{T_1+\eps}L^2} \leq 2R$
as long as $N$ and $T_1+\eps$ satisfy \eqref{46}, \eqref{48}, and \eqref{49}.
Note that these conditions are independent of $m \in \mathbb{N}$.
In this way, we obtain a time interval $[0, T]$
such that $ \|v^{[m]}\|_{C_{T}L^2} \leq 2R$ for all $m\in\mathbb{N}$.

Moreover, by repeating a similar computation on the difference, 
we have
\begin{align} \label{M3}
 \| \G_{v_0^{[m_1]}} v^{[m_1]} - & \G_{v_0^{[m_2]}} v^{[m_2]}\|_{C_T L^2} \notag\\
& \leq 
(1+ \tfrac{1}{10})\|v_0^{[m_1]}-v_0^{[m_2]}\|_{L^2}
+ \tfrac{1}{5}
\|v^{[m_1]}- v^{[m_2]}\|_{C_TL^2}
\end{align}

\noi
by possibly taking larger $N$ and smaller $T$.
Since $v^{[m_j]}$ is a (smooth) solution with initial condition $v_0^{[m_j]}$,
namely $ v^{[m_j]}= \G_{v_0^{[m_j]}} v^{[m_j]}$, 
it follows from \eqref{M3} that
\begin{align} \label{M4}
\|v^{[m_1]}- v^{[m_2]}\|_{C_T L^2}
\leq C'\|v_0^{[m_1]}-v_0^{[m_2]}\|_{L^2}
\end{align}

\noi
for some $C' >0$.
Hence, $\{v^{[m]}\}$ converges in $C([0, T]; L^2)$.

\medskip

Let $v^\infty$ denote the limit.
Next, we  show that $u^\infty := S(t) v^\infty$ satisfies
NLS \eqref{NLS1} on $[0, T]$ in the sense of Definition \ref{DEF:3}.
In the following, we drop $\infty$ in the superscript and simply denote $v^\infty$ and $u^\infty$ by $v$ and $u$.
Also, let $u^{[m]}(t):= S(t) v^{[m]}(t)$, where 
$v^{[m]}$ is the smooth  solution to \eqref{NLS4} with smooth initial condition $v_0^{[m]}$ as above.
Note that $u^{[m]}$ is the smooth  solution to \eqref{NLS1} with smooth initial condition 
$u_0^{[m]} := v_0^{[m]}$.
Moreover, $u^{[m]}$ converges to $u$ in $C([0, T];L^2)$,
since $v^{[m]}$ converges to $v$ in $C([0, T];L^2)$.
Thus, $\dt u^{[m]}$ and $\dx^2 u^{[m]}$
converge to $\dt u$ and $\dx^2 u$ in $\mathcal{D}'(\T\times (0, T))$, respectively.
Since $u^{[m]}$ satisfies \eqref{NLS1} for each $m$, 
we see that 
\[\mathcal{N}(u^{[m]}) : = u^{[m]}|u^{[m]}|^2
= - i \dt u^{[m]} + \dx^2 u^{[m]}\]

\noi
also converges to some distribution $w$ in $\mathcal{D}'(\T\times (0, T))$.

\begin{proposition} \label{PROP:NON2}
Let $w$ be the limit of $\mathcal{N}(u^{[m]})$ in the distributional sense as above.
Then, $w = \mathcal{N}(u)$, where $\mathcal{N}(u)$ on the right hand side is to be interpreted
in the sense of Definition \ref{DEF:2}.
\end{proposition}

 \noi
 We present the proof of Proposition \ref{PROP:NON2} at the end of this section.
Assuming Proposition \ref{PROP:NON2}, we see that $u = u^\infty$
is a solution to \eqref{NLS1} in the extended sense as in Definition \ref{DEF:3}.

\medskip
It follows from  \eqref{46}, \eqref{48},  and \eqref{49}
(with $R = \|v_0\|_{L^2} + 1$ and $\wt{R} = 4R$)
that the time of existence $T$ satisfies $T \gtrsim (1+\|v_0\|_{L^2})^{-\beta}$
for some $\beta > 0$.
The Lipschitz dependence on initial data follows from \eqref{M4},
bypassing smooth approximations.

Lastly, for $s> 0$, we only have to note that all the lemmata in Sections \ref{SEC:2}
and \ref{SEC:3} hold true
even if we replace the $L^2$-norm by the $H^s$-norm.
Indeed, 
if $n^{(j)}$ is large, then 
there exists at least one of $n^{(j)}_1,  n^{(j)}_2$, and $n^{(j)}_3$
satisfies $|n^{(j)}_k| \geq \frac{1}{3}|n^{(j)}|$,
since we have
$n^{(j)} = n^{(j)}_1 - n^{(j)}_2 + n^{(j)}_3$.
Hence, in the estimates for the terms in  the $J$th generation (Lemmata \ref{LEM:N^J+1_0}, \ref{LEM:N^J+1_r},
\ref{LEM:N^J+1_1}), 
there exists at least one frequency $n^{(j)}_k$ (with some $j = 1, \dots, J$) such that 
\[\jb{n}^s \leq 3^{js} \jb{n^{(j)}_k}^s \leq 3^{Js} \jb{n^{(j)}_k}^s.\]

\noi
Note that the constant grows exponentially in $J$.
However, this exponential growth does not cause a problem
thanks to the factorial decay in the denominator 
(as seen in the proof Lemma \ref{LEM:N^J+1_0}.)

\begin{proof} [Proof of Proposition \ref{PROP:NON2}]

Let $\{T_N\}_{N\in \mathbb{N}}$ be a sequence of Fourier cutoff multipliers as in Definition \ref{DEF:1}.
Fix a test function on $\T\times (0, T)$.
Then, we need to show that 
given $\eps > 0$, there exists $N_0$ such that for all $N \geq N_0$, we have
\begin{align} \label{NON3}
|\jb{ w - \N(T_N u), \phi }| <\eps.
\end{align}

\noi
Write the left hand side of \eqref{NON3} as
\begin{align*}
|\jb{ w - \N(T_N u), \phi }|
\leq \   |\jb{ w  - \N(u^{[m]}), &  \phi }|
 + |\jb{ \N(u^{[m]})  - \N(T_N u^{[m]} ), \phi }| \notag \\
& + |\jb{\N(T_N u^{[m]} ) - \N(T_N u ), \phi}|.
\end{align*}

\noi
By definition of of $w$, we see that 
\begin{equation}  \label{NON4}
|\jb{ w - \N( u^{[m]}), \phi }| <\tfrac{1}{3}\eps
\end{equation}

\noi
for sufficiently large $m\in \mathbb{N}$. 

Next, consider the second term for {\it fixed} $m$.
By writing 
$\N(u^{[m]})  - \N(T_{N} u^{[m]})$ in a telescoping sum, 
we only consider
\begin{equation*}
|\jb{\N\big((I - T_{N_1}) u^{[m]}, u^{[m]}, u^{[m]}\big), \phi}|,
\end{equation*}

\noi
where $\mathcal{N}(u_1, u_2, u_3) = u_1 \cj{u_2} u_3$
and $I$ denotes the identity operator.
(The other terms in the telescoping sum have similar forms.)
By H\"older inequality and Sobolev embedding, we obtain
\begin{align} \label{NON5}
|\jb{\N\big((I - T_{N_1}) u^{[m]}, & u^{[m]}, u^{[m]}\big), \phi}|
 \leq \|\phi\|_{L^2_{x, T}} \| u^{[m]}\|^2_{L^\infty_{x, T}} 
\|(I - T_N) u^{[m]}   \|_{L^2_{x, T}} \notag \\
& \leq C_\phi \| u^{[m]}\|^2_{C_T H^{\frac{1}{2}+}} 
\|(I - T_N) u^{[m]}   \|_{L^2_{x, T}} \notag \\
& \leq C_{\phi, m} \|(I - T_N) u^{[m]}   \|_{L^2_{x, T}}, 
\end{align}

\noi
where $L^2_{x, T}$ denotes $L^2(\T \times [0, T])$.
Here, we used the fact that 
$\| u^{[m]}\|^2_{C_T H^{\frac{1}{2}+}}$ is a finite constant
(depending on $m$.)
By definition of the Fourier cutoff operators, 
$\big((I - T_N) u^{[m]}\big)^\wedge(n, t)$ converges pointwise in $n$ and $t$.
Then, by Dominated Convergence Theorem, there exists $N_0 = N_0(m)$
such that 
\begin{equation} \label{NON6}
\eqref{NON5} <\tfrac{1}{3}\eps.
\end{equation}

\noi
for all $N \geq N_0$.

As for the third term, 
first consider the sequence $\{\N(T_N u^{[m]})\}_{m \in \mathbb{N}}$
for each fixed $N$.
By applying the Poincar\'e-Dulac normal form reduction 
to $\{S(-t) \N(T_N u^{[m]})\}_{m \in \mathbb{N}}$
(which is basically the nonlinearity in the $v$-equation \eqref{NLS4} modulo $T_N$)
 as in Sections \ref{SEC:2} and \ref{SEC:3},
we see that $\{ \N(T_N u^{[m]})\}_{m \in \mathbb{N}}$ is a Cauchy sequence in $\mathcal{D}'(\T\times (0, T))$,
as $m \to \infty$ for each fixed $N$
since $u^{[m]}$ is Cauchy in $C([0, T];L^2)$.
Moreover, this convergence is uniform in $N$ since the multipliers for $T_N$ 
are uniformly bounded in $N$.

On the other hand, note that for fixed $N$, 
$T_N u $ is in $C_T H^\infty$, since the multiplier $m_N$ for $T_N$
has a compact support.
Thus,  
$\N(T_N u ) = T_N u |T_N u|^2$ makes sense as a function. 
Hence, for {\it fixed} $N$, we can choose $m$ large such that
\begin{align*} 
|\jb{\N(T_N u^{[m]} ) - \N(T_N u ), \phi}|
& \leq \|\phi\|_{L^4_{x, T}} \big(\|T_N u^{[m]}\|^2_{L^4_{x, T}} + \|T_N u\|_{L^4_{x, T}}^2\big) 
\|T_N u^{[m]}  - T_N u \|^2_{L^4_{x, T}} \notag \\
& \leq C_{\phi, \|u\|_{C_TL^2}} M^\frac{3}{4} T^\frac{3}{4}
\| u^{[m]}  -  u \|^2_{C_T L^2} <  \tfrac{1}{3}\eps,
\end{align*}

\noi
by Sobolev inequality, where $M = M(N)\in \mathbb{N}$ is chosen such that $\supp (m_N) \subset [-M, M]$.
i.e. $\N(T_N u^{[m]} )$ converges to $\N(T_N u )$
in $\mathcal{D}'(\T\times (0, T))$
as $m \to \infty$ for each fixed $N$.

Combining these two observations, 
we conclude that  $\N(T_N u^{[m]} )$ converges to $\N(T_N u )$
in $\mathcal{D}'(\T\times (0, T))$
as $m \to \infty$ {\it uniformly} in $N$.
Namely, 
\begin{equation} \label{NON7}
|\jb{\N(T_N u^{[m]} ) - \N(T_N u ), \phi}| <\tfrac{1}{3}\eps
\end{equation}

\noi
for all sufficiently large $m$, uniformly in $N$.
Therefore, \eqref{NON3} follows 
by first choosing $m$ sufficiently large such that  \eqref{NON4} and  \eqref{NON7} hold, 
then choosing $N_0 = N_0(m)$ such that \eqref{NON6} holds.
\end{proof}

\section{Unconditional uniqueness in $C_tH^s$, $s\geq \frac{1}{6}$}
\label{SEC:5}

In this section, we prove Theorem \ref{thm1}.
More precisely, we justify the formal computations in Sections \ref{SEC:2} and \ref{SEC:3}
on the additional regularity assumption.
Then, the Lipschitz bound implies the uniqueness.
In the following, we justify our computations,
assuming that $u$ is a solution to \eqref{NLS1} in $C([0, T]; L^3(\T))$.

\medskip
First, we make sense of the use of $\dt v_n (t) = e^{-in^2t} \ft{u}(n, t)$ in Sections \ref{SEC:2} and \ref{SEC:3}.
Suppose that $u  \in C([0, T]; L^3(\T))$.
Then, we have $u|u|^2 \in C([0, T]; L^1(\T))$,
and hence $\mathbb{P}_{\leq M} (u|u|^2) \in C([0, T]; H^\infty(\T))$
for any $M \in \mathbb{N}$,
where $\mathbb{P}_{\leq M}$ is the Dirichlet projection onto the frequencies $|n|\leq M$.
This implies $(\dx^2 \mathbb{P}_{\leq M}u)^\wedge \in C([0, T]; H^\infty(\T))$.
Hence, from the equation \eqref{NLS1}, we see that 
$(\mathbb{P}_{\leq M} \dt u)^\wedge \in C([0, T]; H^\infty(\T))$.
In particular, $\ft{ u}(n, \cdot)$ is a $C^1$-function in $t$.

\medskip
In Sections \ref{SEC:2} and \ref{SEC:3},  we switched the order of summation and the time differentiation.
For example, see \eqref{N12}.
This can be justified, also by assuming $u \in C([0, T]; L^3(\T))$.
First, we state a lemma.
\begin{lemma}\label{LEM:CONV1}
Let $\{f_n\}_{n\in \mathbb{N}}$ be a sequence in $\mathcal{D}'_t$.
Suppose that $\sum_n  f_n$ converges (absolutely) in $\mathcal{D}'_t$.
Then, $\sum_n\dt f_n$ converges (absolutely) in $\mathcal{D}'_t$
and $\dt (\sum_n f_n ) = \sum_n \dt f_n$.
\end{lemma}

\begin{proof}
Recall that a sequence of distribution $g_n$ is said to converge to a distribution $g$
if, for all $\phi \in \mathcal{D}$, we have $\jb{g_n, \phi} \to \jb{g, \phi}$.
Thus, we have 
\begin{align*}
\Big\langle\sum_{n= 1}^\infty f'_n, \phi\Big\rangle
= \lim_{N \to \infty}
\Big\langle\sum_{n= 1}^N f'_n, \phi\Big\rangle
= \lim_{N \to \infty}
\sum_{n= 1}^N \jb{f'_n, \phi},
\end{align*}

\noi
if the right hand side exists.
By the definition of a distributional derivative, we have
\begin{align*}
 \lim_{N \to \infty}
\sum_{n= 1}^N \jb{f'_n, \phi}
= - \lim_{N \to \infty}
\sum_{n= 1}^N \jb{f_n, \phi'},
%
\end{align*}

\noi
where the right hand side converges since $\sum_n f_n$ converges in $\mathcal{D}'$.
Hence, $\sum_n f'_n$ converges in $\mathcal{D}'$.
The second claim follows once we note the following.
\begin{align*}
- \lim_{N \to \infty}
\sum_{n= 1}^N \jb{f_n, \phi'}
= - \lim_{N \to \infty}
\Big\langle\sum_{n= 1}^N f_n, \phi'\Big\rangle
= - \Big\langle\sum_{n= 1}^\infty f_n, \phi'\Big\rangle
= \Big\langle \dt \Big(\sum_{n= 1}^\infty f_n\Big), \phi\Big\rangle,
\end{align*}

\noi
where the second equality follows from the definition of $\sum_n f_n$
as a distributional limit.
\end{proof}

Now, we consider \eqref{N12} for fixed $n$.
Then, we want to apply Lemma \ref{LEM:CONV1} to a sequence
\[ \{ a_{n, n_1, n_2}(t) \}:=\bigg\{  \frac{e^{- i \Phi(\bar{n})t } }{2(n - n_1) (n - n_3)}
v_{n_1}(t) \cj{v}_{n_2}(t)v_{n_3}(t)\bigg\}, \]

\noi
where $n = n_1 - n_2 + n_3$ and $(n, n_1, n_2, n_3) \in A(n)^c$.
(Here, $a_{n, n_1, n_2}$ depends on several indices with a restriction
(i.e. on $A_N(n)^c$), but we can arrange them to be a sequence.)

By Lemma \ref{LEM:N21}, 
$\sum_{A_N(n)^c} a_{n, n_1, n_2} $ converges absolutely and is bounded in $C([0, T])$
(for fixed $n$.)
In particular, for each $n_1$ and $n_2$, 
$a_{n, n_1, n_2}$ is a distribution on $[0, T]$.
By Lemma \ref{LEM:CONV1}, we have
\begin{align}
 \dt \bigg[
 \sum_{ A_N(n)^c} 
& \frac{e^{- i \Phi(\bar{n})t } }{- i \Phi(\bar{n})}
v_{n_1}  \cj{v}_{n_2}v_{n_3}\bigg]  
=   \sum_{ A_N(n)^c} 
 \dt \bigg[
\frac{e^{- i \Phi(\bar{n})t } }{- i \Phi(\bar{n})}
v_{n_1} \cj{v}_{n_2}v_{n_3}\bigg]   \notag \\
& =   \sum_{ A_N(n)^c} 
\dt\bigg( \frac{e^{- i \Phi(\bar{n})t } }{-i\Phi(\bar{n})}\bigg)
v_{n_1} \cj{v}_{n_2}v_{n_3} 
+   \sum_{A_N(n)^c} 
\frac{e^{- i \Phi(\bar{n})t } }{-i\Phi(\bar{n})}
\dt\big( v_{n_1} \cj{v}_{n_2}v_{n_3}\big).  \label{CONV2}
\end{align}

\noi
In the second equality, we applied the product rule.
It is in this step that we needed the additional regularity
$u \in C([0, T]; L^3)$
so that $v_n$ is continuously differentiable and the product rule is applicable.
A similar argument justifies the exchange of the sum and the time differentiation 
in the $J$th generation.  We omit the details.

\medskip
Indeed, if we assume that $u \in C([0, T]; L^3)$, we can say more on this issue.
First, note that from \eqref{NLS4}, we have $\dt v_n = e^{-in^2t} \N(S(-t) v)_n = e^{-in^2t} \N(u)_n$.
Then, 
\begin{equation} \|\dt v_n\|_{C_T \ell^\infty_n} =  \|\N(u)_n\|_{C_T \ell^\infty_n}
\leq \|\N(u)\|_{C_T L^1_x} \leq \|u\|^3_{C_TL^3_x}.
\label{CONV3}
\end{equation}

\noi
Hence, $\dt v_n \in C([0, T]; \ell^\infty_n)$.
In the following, fix $n$.
Then, by a variant of Lemma \ref{LEM:N21}, the second term on the right hand side of \eqref{CONV2}
is estimated as 
\begin{align}
\bigg|   \sum_{A_N(n)^c} 
\frac{e^{- i \Phi(\bar{n})t } }{-i\Phi(\bar{n})}
\dt\big( v_{n_1} \cj{v}_{n_2}v_{n_3}\big)\bigg|
\lesssim \|\dt v_n\|_{C_T \ell^\infty_n}\|v\|_{C_T L^2}^2
\leq \|u \|_{C_T L^3}^3\|v\|_{L^2}^2,
\end{align}

\noi
where the convergence is absolute and uniform (in $t$.)
The first term in \eqref{CONV2} can be written as $ e^{-in^2t} \N(u)_n$
and thus also converges in view of \eqref{CONV3}.
(Here, the convergence is not absolute, 
but uniform in $t$.) 
Therefore, we can simply switch the sum and the time differentiation 
(i.e. the first equality in \eqref{CONV2}) in classical sense.
The argument for the $J$th generation is similar and we omit the details.

\medskip

Lastly, the regularity $C([0, T]; L^3)$ was sufficient to justify the formal
computations in Sections \ref{SEC:2} and \ref{SEC:3}.
However, in order to prove the lemmata, which are proven on the Fourier side,
we need a $L^2$-based space of the same scaling, namely $C([0, T]; H^\frac{1}{6})$.

\medskip

\noindent {\bf Acknowledgments:} The authors would like to thank Prof.~Nicolas Burq,  Prof.~Herbert Koch, 
and Prof. Nader Masmoudi for the helpful comments and discussions.

\end{document}